\documentclass[12pt,a4paper,oneside]{article}
\usepackage[utf8]{inputenc}
\usepackage{amsmath}
\usepackage{amsfonts}
\usepackage{amssymb}
\usepackage{graphicx}
\usepackage[english]{babel}
\usepackage{color}
\usepackage{tikz}
\usepackage[T1]{fontenc}
\usepackage{nicefrac}
\usepackage{cite}
\usepackage{dsfont}
\usepackage[symbol]{footmisc}
\usepackage{hyperref}

\definecolor{kerstin}{RGB}{25,180,60}

\newcommand{\re}{\mathbb{R}}
\newcommand{\na}{\mathbb{N}}

\newcommand{\eps}{\varepsilon}

\newcommand{\diver}{\operatorname{div}}
\newcommand{\lz}{{L^2(D)}}
\newcommand{\HS}{\text{HS}}
\newcommand{\sobolevp}{W^{1,p}_0(D)}

\newcommand{\erww}[1]{\mathbb{E}\left[#1\right]}

\newcommand{\halbe}{\frac{1}{2}}

\newcommand{\ranglev}{\rangle_{q',q}}

\newcommand{\supt}{\sup_{t\in[0,T]}}
\newcommand{\knorm}{\|k\|_{L^2(D\times D)}}
\newcommand{\erwws}[1]{\mathbb{E}'\left[#1\right]}
\newcommand{\mP}{\mathbb{P}}
\newcommand{\tiT}{{t\in[0,T]}}

\newcommand{\sobolevmq}{W^{m,q}_0(D)}
\newcommand{\sobolevmqs}{W^{-m,q'}(D)}
\newcommand{\mY}{\mathcal{Y}}
\newcommand{\mW}{\mathcal{W}}
\newcommand{\vnd}{v_n}

\newcommand{\mF}{\widetilde{\mathcal{F}}}
\newcommand{\mFt}{(\widetilde{\mathcal{F}}_t)_{t\in[0,T]}}

\newtheorem{defi}{Definition}[section]

\newtheorem{lem}[defi]{Lemma}
\newtheorem{teo}[defi]{Theorem}
\newtheorem{prop}[defi]{Proposition}
\newtheorem{cor}[defi]{Corollary}

\newtheorem{remark}[defi]{Remark}

\newenvironment{proof}{\noindent{\textit{Proof.}}}{\hfill$\square$}
\newenvironment{notation}{\noindent{\textit{Notation.}}}

\numberwithin{equation}{section}

\begin{document}
\title{Well-posedness of stochastic evolution equations with Hölder continuous noise}

\author{Kerstin Schmitz\footnotemark[1], \and Aleksandra Zimmermann\footnotemark[1]}
\date{}

\footnotetext[1]{TU Clausthal, Institut f\"ur Mathematik, Clausthal-Zellerfeld, Germany, kerstin.schmitz@tu-clausthal.de,
aleksandra.zimmermann@tu-clausthal.de}

\maketitle

\begin{abstract}
We show existence and pathwise uniqueness of probabilistically strong solutions to a pseudomonotone stochastic evolution problem on a bounded domain $D\subseteq\re^d$, $d\in\na$, with homogeneous Dirichlet boundary conditions and random initial data $u_0\in L^2(\Omega;L^2(D))$. The main novelty is the presence of a merely H{\"o}lder continuous multiplicative noise term. In order to show the well-posedness, we simultaneously regularize the H{\"o}lder noise term by inf-convolution and add a perturbation by a higher order operator to the equation. Using a stochastic compactness argument we may pass to the limit and we obtain first a martingale solution. Then by a pathwise uniqueness argument we get existence of a probabilistically strong solution.
\end{abstract}



\tableofcontents

\section{Introduction}

\subsection{Statement of the problem, motivation, and former results}

Let $T>0$, $D\subseteq\re^d,\ d\in\na,$ be a bounded domain, and $(\Omega,\mathcal{A},\mathbb{P})$ a probability space endowed with a right-continuous, complete filtration $(\mathcal{F}_t)_{0\leq t\leq T}$.
We want to show existence and uniqueness of strong solutions to stochastic $p$-Laplace evolution equations of the form
\begin{align}\label{spde}
\begin{aligned}
du-\diver a(x,u,\nabla u)\, dt+f(u)\,dt&=B(t,u)\,dW_t &&\text{in }\Omega\times[0,T]\times D\\
u &= 0 &&\text{on }\Omega\times[0,T]\times\partial D\\
u(0,\cdot)&= u_0 &&\text{in }\Omega\times D,
\end{aligned}
\end{align}
where $u_0$ is assumed to be in $L^{2}(\Omega;\lz)$ and $\mathcal{F}_0$-measurable.
We fix a separable Hilbert space $U$ such that $U\supseteq L^2(D)$ and a symmetric, non-negative trace class operator $Q:U\rightarrow U$ with $Q^\halbe(U)=L^2(D)$. We endow $U$ with an orthonormal basis of eigenvectors of $Q$. In the following let $(W_t)_{0\leq t \leq T}$ be a $(\mathcal{F}_t)_{0\leq t\leq T}$-adapted $Q$-Wiener process with values in $U$.
The integral on the right-hand side of \eqref{spde} is understood in the sense of Itô.
The function $f\in L^\infty(\re)$ is assumed to be Lipschitz continuous with Lipschitz constant $L_f\geq0$ and satisfies $f(0)=0$.\\
For $a:D\times\re\times\re^d\rightarrow\re^d$ we assume, that it is a Carathéodory function and satisfies for $\max\{1,\frac{2d}{d+2}\}<p<\infty$ the usual Leray-Lions conditions which will be specified in Section~\ref{Sec_Hypo}.

Our aim is to show existence and pathwise uniqueness of probabilistically strong solutions to \eqref{spde}. The classical monotonicity method to show well-posendess is originated in \cite{L69} for deterministic equations. This method was extended to stochastic partial differential equations by Pardoux (see \cite{P75}) and was generalized by Krylov and Rozovskii (see \cite{KR07}) and Liu and Röckner in \cite{LR15}.
Key properties for these well-posedness results are certain monotonicity, coercivity and growth conditions of the (locally) monotone operator in combination with the noise term. These assumptions have been applied and extended by many authors (see, e.g., \cite{GRZ09,LR10,L11,L13,LR13,BDR16,GHV22}).\\
The main task we want to tackle in this study is the presence of a pseudomonotone operator and a merely H\"older continuous multiplicative noise term. Precisely, we assume that the operator $B:(0,T)\times\lz\rightarrow\HS(\lz)$  is Hölder continuous but not necessarily Lipschitz continuous in its second variable, where $\HS(\lz)$ denotes the space of Hilbert-Schmidt operators from $\lz$ to $\lz$.\\
There are many results that address linear SPDEs with Hölder continuous noise or, more generally, nonlinear SPDEs with Lipschitz noise and Hölder continuous coefficients in the literature. Let us mention the results on existence and uniqueness of solutions to stochastic Volterra equations with non-Lipschitz coefficients and on stochastic evolution equations with non-Lipschitz coefficients in \cite{W08} and \cite{Z09}.
Existence of mild solutions to the stochastic heat equation with Hölder diffusion coefficients is well-known and has been studied in (see \cite{MMP14, MP11, MPS06, S94}).
The question of uniqueness of solutions to the stochastic heat equation with non-Lipschitz diffusion coefficient and space-time white noise as well as colored noise was studied in \cite{MP11,MMP14,MN15}.
In these contributions, a semigroup approach is available and allows to use the framework of mild solutions. Motivated by this results, our aim is to study existence and uniqueness for evolution equations driven by nonlinear pseudomonotone operators and non-Lipschitz multiplicative noise in the variational framework. \\
In the recent contribution \cite{RSZ22}, well-posedness of SPDEs driven by multiplicative noise with fully local monotone coefficients has been considered. The authors use Galerkin approximations for the proof of existence of probabilistically weak solutions and a refined $L^2$-technique for the proof of pathwise uniqueness.
The results in our contribution differ twofold from the results in \cite{RSZ22}. Firstly, we use different techniques, namely the simultaneous perturbation with a higher-order operator and regularization by inf-convolution in the noise. Secondly, our operator is rather pseudomonotone than locally monotone and may therefore not satisfy the local monotonicity conditions from \cite{RSZ22}, see Remark \ref{240215_01} for more details.

To show existence of strong solutions to \eqref{spde} we approximate the non-Lipschitz operator $B$ by a Lipschitz-continuous operator. In addition, we adapt the ideas in \cite{STVZ23} and add a singular perturbation in form of a higher order operator to the equation. This enables us to apply the well-posedness result stated in \cite{LR15} to get a variational solution to the approximated equation.
To obtain then a martingale solution to \eqref{spde} we use a stochastic compactness argument of Prokhorov and Skrorokhod which is classical in the framework of SPDEs and has been used in, e.g., \cite{DG22,DKL23,BNSZ23, FG95,DGT11,BBNP14,BBNP14book,DHV16,SWZ19,VZ19,BV19,VZ21,RSZ22}, see also \cite{BFH18} for a more extensive list of references. 
Existence of a probabilistically strong solution to \eqref{spde} follows from a pathwise uniqueness argument of Gyöngy and Krylov (see \cite{GK96}). 

\subsection{Hypotheses}\label{Sec_Hypo}

For $a:D\times\re\times\re^d\rightarrow\re^d$ we assume, that it is a Carathéodory function, i.e. $D\ni x\mapsto a(x,\lambda,\xi)$ is  measurable for all $(\lambda,\xi)\in\re\times\re^d$ and $\re\times\re^d\ni (\lambda,\xi)\mapsto a(x,\lambda,\xi)$ is continuous for a.e. $x\in D$. Moreover $a$ satisfies for $\max\{1,\frac{2d}{d+2}\}<p<\infty$
\begin{itemize}
\item[(A1)] $(a(x,\lambda,\xi)-a(x,\lambda,\eta))(\xi-\eta)\geq 0$ for all $\xi,\eta\in\re^d,\lambda\in\re$ and a.e. $x\in D$
\item[(A2)] There exist $\kappa\in L^1(D$), constants $C_1>0,\ C_2,C_3,C_4\geq0$, $1\leq\nu<p$, and a non-negative function $g\in L^{p'}(D)$ such that for all $(\lambda,\xi)\in\re\times\re^d$ and a.e. $x\in D$ there holds
\begin{align*}
a(x,\lambda,\xi)\cdot\xi\geq\kappa(x)+C_1|\xi|^p-C_2|\lambda|^\nu
\end{align*}
and
\begin{align*}
|a(x,\lambda,\xi)|\leq C_3|\xi|^{p-1}+C_4|\lambda|^{p-1}+g(x).
\end{align*}
\item[(A3)] There exist a constant $C_5\geq0$ and a non-negative function $h\in L^{p'}(D),$ such that for all $\lambda_1,\lambda_2\in\re,\xi\in\re^d$ and a.e. $x\in D$ there holds
\begin{align*}
|a(x,\lambda_1,\xi)-a(x,\lambda_2,\xi)|\leq(C_5|\xi|^{p-1}+h(x))|\lambda_1-\lambda_2|.
\end{align*}
\end{itemize}

\begin{remark}\label{240215_01}
For $p\geq 2$, an operator induced by a Carath\'{e}odory function satisfying (A1)-(A3) is a slight generalization of the operator 
\[\mathcal{A}:W^{1,p}_0(D)\rightarrow W^{-1,p'}(D), \ u\mapsto \mathcal{A}(u)=-\Delta_p(u)+\diver F(u),\] 
where $\Delta_p(u)=\diver(|\nabla u|^{p-2}\nabla u)$ and $F:\re\rightarrow\re^d$ is Lipschitz continuous with $F(0)=0$.
\end{remark}

For $\sigma:(0,T)\times \mathbb{R}\rightarrow \mathbb{R}$ we assume that
\begin{itemize}
\item[(S1)] For a.e. $t\in (0,T)$
\[\re\ni\lambda\mapsto\sigma(t,\lambda)\]
is continuous and
$$(0,T)\ni t \mapsto\sigma(t,\lambda)$$
is measurable for every $\lambda\in\re$.
\item[(S2a)] $\sigma$ is $\alpha$-Hölder continuous, i.e., there exists an $\alpha\in(0,1]$ and a constant $L_\alpha>0$ such that for all $\lambda,\mu\in\re$, a.e. $t\in (0,T)$ there holds
$$|\sigma(t,\lambda)-\sigma(t,\mu)|\leq L_\alpha|\lambda-\mu|^\alpha.$$
\item[(S2b)] We assume $\sigma(t,0)=0$ for almost all $t\in (0,T)$.
\item[(S3)] $\sigma$ has a sublinear growth, i.e., there exists $C_\sigma>0$ such that
$$|\sigma(t,\lambda)|^2\leq C_\sigma(1+|\lambda|^2)$$
for all $\lambda\in\re$, a.e. $t\in (0,T)$.
\end{itemize}

In the following we introduce the notion of infinite dimensional Hölder noise. Let $\HS(\lz)$ denote the space of Hilbert-Schmidt operators from $L^2(D)$ to $L^2(D)$.
We consider an operator $B:(0,T)\times\lz\rightarrow \HS(\lz)$ that is for $(t,v)\in (0,T)\times\lz$, $\varphi\in\lz$ of the form
\begin{align}\label{defi_B}
B(t,v)\varphi(x)=\sigma(t,v(x))\int_D k(x,y)\varphi(y)\,dy
\end{align}
for a.e. $x\in D$, with a symmetric kernel $k\in L^2(D\times D)$ that satisfies
$$\operatorname*{ess\,sup}_{y\in D} \|k(\cdot,y)\|_{L^2(D)}^2=\operatorname*{ess\,sup}_{x\in D}\|k(x,\cdot)\|_{L^2(D)}^2 \leq C_k$$
for a constant $C_k\geq0$.

\begin{remark}
Let $X,Y$ be two Hilbert spaces. A bounded operator $K:X\rightarrow Y$ is a Hilbert-Schmidt operator iff it is a Hilbert-Schmidt integral operator, i.e. there exists a kernel $l\in L^2(X\times Y)$ such that for all $\varphi\in\lz$ and a.e. $x\in D$
\begin{align*}
    K\varphi(x)=\int_D l(x,y)\varphi(y)\,dy.
\end{align*}
Moreover, there holds $\|K\|_{\HS}=\|l\|_{L^2(D)}$, see \cite[Satz 3.19]{W00} and \cite[p.93]{LR17}.
\end{remark}

\begin{notation}
In the following we will denote by $\|\cdot\|_r$ the norm in $L^r(D)$ for $1\leq r\leq\infty$, by $\langle\cdot,\cdot\rangle_{L^2}$ the dual pairing in $L^2(D)$, and by $\|\cdot\|_{\HS}$ the Hilbert-Schmidt norm on $\HS(L^2(D))$.\\
\end{notation}

The operator $B$ is well defined on $L^2(D)$, indeed, for $v,\varphi\in\lz$ and a.e. $t\in (0,T)$ there holds by using Cauchy-Schwarz inequality, (S3), and Fubini's theorem
\begin{align*}
\|B(t,v)(\varphi)\|_2^2&=\int_D\left|\int_D\sigma(t,v(x))k(x,y)\varphi(y)\,dy\right|^2\,dx\\
&\leq\int_D\left(\int_D|\sigma(t,v(x))|^2|k(x,y)|^2\,dy\right)\left(\int_D|\varphi(y)|^2\,dy\right)\,dx\\
&\leq\|\varphi\|_2^2\int_D\int_D C_\sigma(1+|v(x)|^2)|k(x,y)|^2\,dy\,dx\\
&\leq \|\varphi\|_2^2\left(C_\sigma\|k\|_{L^2(D\times D)}^2+C_\sigma\int_D\int_D|v(x)|^2|k(x,y)|^2\,dy\,dx\right)\\
&=\|\varphi\|_2^2C_\sigma\left(\knorm^2+\int_D|v(x)|^2\|k(x,\cdot)\|_2^2\,dy\right)\\
&\leq\|\varphi\|_2^2C_\sigma\left(\knorm^2+\operatorname*{ess\,sup}_{x\in D}\|k(x,\cdot)\|_2^2\|v\|_2^2\right).
\end{align*}
Let $(e_n)_{n\in\na}$ be an orthonormal basis of $\lz$.
Using Parseval's identity and (S3) we obtain for $v\in\lz$ and a.e. $(\omega,t)\in\Omega\times(0,T)$ by an analogous argumentation as above
\begin{align}\label{B_bounded}
\begin{aligned}
\|B(t,v)\|_\HS^2&=\sum_{n\in\na}\|B(t,v)(e_n)\|_2^2=\sum_{n\in\na}\int_D\left|\langle\sigma(t,v(x))k(x,\cdot),e_n(\cdot)\rangle_{L^2}\right|^2\,dx\\
&=\int_D|\sigma(t,v(x))|^2\sum_{n\in\na}\left|\langle k(x,\cdot),e_n(\cdot)\rangle_{L^2}\right|^2\,dx=\int_D|\sigma(t,v(x))|^2\|k(x,\cdot)\|_2^2\,dx\\
&\leq\int_D C_\sigma(1+|v(x)|^2)\|k(x,\cdot)\|_2^2\,dx\\
&\leq C_\sigma(\knorm^2+C_k\|v\|_2^2).
\end{aligned}
\end{align}
Using (S2) instead of (S3) in the same manner we get for $v,w\in L^2(D)$, and a.e. $t\in (0,T)$
\begin{align}\label{estimateBsigma}
\begin{aligned}
\|B(t,v)-B(t,w)\|_{\HS}^2
&\leq \int_D|\sigma(t,v(x))-\sigma(t,w(x))|^{2}\int_D|k(x,y)|^2\,dy\,dx\\
&\leq L_\alpha^2\int_D|v(x)-w(x)|^{2\alpha}\|k(x,\cdot)\|_2^2\,dx\\
&\leq C_k L_\alpha^2\|v-w\|_{2\alpha}^{2\alpha}\leq C_k L_\alpha^2 C(\alpha)\|v-w\|_2^{2\alpha}.
\end{aligned}
\end{align}
where $C(\alpha)>0$ is a constant from the continuous embedding of $L^2(D)$ in $L^{2\alpha}(D)$ for $\alpha\in (0,1)$.\\

\subsection{Main results and outline}

Let the Assumptions of Section~\ref{Sec_Hypo} hold.

\begin{defi}\label{defsolution}
Let $f:\re\rightarrow\re$ be a given Lipschitz continuous function, $\sigma:(0,T)\times\re\rightarrow\re$ fulfills (S1)-(S3) for $\alpha\in(0,1]$.
\begin{itemize}
\item[$i)$] A stochastic process $u$ is called a probabilistically strong solution, if
\begin{enumerate}
    \item $u$ is an $(\mathcal{F}_t)_{\tiT}$-adapted stochastic process and $u\in L^2(\Omega;C([0,T];L^2(D)))\cap L^p(\Omega;L^p(0,T;W^{1,p}_0(D)))$
    \item $u(0)=u_0$ $\mP$-a.s. in $\Omega$
    \item there holds for all $\tiT$, in $L^2(D)$, $\mP$-a.s. in $\Omega$
    \begin{align*}
    u(t)-u_0-\int_0^t a(x,u(s),\nabla u(s))\,ds+\int_0^t f(u(s))\,ds=\int_0^t B(s,u(s))\,dW_s.
    \end{align*}    
\end{enumerate}
\item[$ii)$] A triple $\left((\Omega',\mathcal{A}',\mFt,\mP'),\widetilde{u},(\mW_t)_{\tiT}\right)$ is called a martingale solution to \eqref{spde} with initial value $v_0$, if
    \begin{enumerate}
        \item $(\Omega',\mathcal{A}',\mFt,\mP')$ is a stochastic basis with a complete, right-continuous filtration
        \item $(\mW_t)_{\tiT}$ is an $\mFt$-adapted $Q$-Wiener process on $(\Omega',\mathcal{A}',\mP')$
        \item $\widetilde{u}\in L^2(\Omega';C([0,T];L^2(D)))\cap L^p(\Omega';L^p(0,T;W^{1,p}_0(D)))$ is an $\mFt$-adapted stochastic process 
        \item $v_0\in L^2(\Omega';L^2(D))$ has the same law as $u_0$
        \item There holds for all $\tiT$, in $L^2(D)$, $\mP'$-a.s. in $\Omega'$
            \begin{align*}
            \widetilde{u}(t)-v_0-\int_0^t a(x,\widetilde{u}(s),\nabla \widetilde{u}(s))\,ds+\int_0^t f(\widetilde{u}(s))\,ds=\int_0^t B(s,\widetilde{u}(s))\,d\mW_s.
            \end{align*}
    \end{enumerate}
\end{itemize}
\end{defi}

\begin{teo}\label{maintheorem_existence}
Assume that $f\in L^\infty(\re)$ is a given Lipschitz continuous function, and $\sigma:(0,T)\times\re\rightarrow\re$ fulfills (S1)-(S3) for an arbitrary $\alpha\in(0,1)$. Then, \eqref{spde} admits a martingale solution in the sense of Definition~\ref{defsolution}~$ii)$.
\end{teo}

\begin{teo}\label{maintheorem_uniqueness}
Assume that $(W_t)_{\tiT}$ is a $(\mathcal{F}_t)_{\tiT}$-adapted $Q$-Wiener process with values in $U$ with respect to the stochastic basis $(\Omega,\mathcal{A},(\mathcal{F}_t)_{\tiT},\mP)$, $f:\re\rightarrow\re$ is a given Lipschitz continuous function, and $\sigma:(0,T)\times\re\rightarrow\re$ fulfills (S1)-(S3) for $\alpha\in[\halbe,1)$. If $u_1,u_2$ are both probabilistically strong solutions to \eqref{spde} in the sense of Definition~\ref{defsolution}~$i)$ with initial values $u_0^1,u_0^2$ in $L^2(D)$ respectively, then there holds for any $t\in[0,T]$
\begin{align*}
    \erww{\|u_1(t)-u_2(t)\|_1}\leq e^{LT}\erww{\|u_0^1-u_0^2\|_1}.
\end{align*}
\end{teo}

\begin{teo}\label{maintheorem}
Assume that $f\in L^\infty(\re)$ is a given Lipschitz continuous function, and $\sigma:(0,T)\times\re\rightarrow\re$ fulfills (S1)-(S3) for $\alpha\in(\halbe,1)$. Then, \eqref{spde} admits a unique probabilistically strong solution $u$ in the sense of Definition~\ref{defsolution}~$i)$.
\end{teo}

\begin{remark}
Theorem~\ref{maintheorem} is a direct consequence of Theorem~\ref{maintheorem_existence} and Theorem~\ref{maintheorem_uniqueness} by an argument of Gyöngy and Krylov, see \cite{GK96}.
\end{remark}

To prove this result we proceed in the following way: Firstly, we approximate in Section \ref{Sec_Noise_appr} the Hölder continuous operator $B$ by a Lipschitz continuous operator $B_n$ and study its properties. Then, in Section \ref{Sec_ApprEqu} we add a higher order operator to the equation which will disappear in the limit afterwards and show existence of a unique variational solution to the approximated equation by using \cite{LR15}. Then, we do some a priori estimates for the approximated solution in Section \ref{Sec_A_priori} which will be useful to pass to the limit later and to show tightness in Section \ref{Sec_Tightness}. These tightness results allow us to apply the stochastic compactness argument of Prokhorov and Skorokhod in Section \ref{Sec_Stoch_Compactness} and obtain thereby a martingale solution to \eqref{spde}. In Section \ref{Sec_Uniqueness} we show pathwise uniqueness of solutions to \eqref{spde} and therefore obtain a strong solution to \eqref{spde} by \cite{GK96}.

\section{Existence of approximate solutions}

\subsection{Lipschitz continuous approximation of the noise}\label{Sec_Noise_appr}

We start by approximating the function $\sigma$ to get a Lipschitz continuous approximation of the operator $B$.

\begin{defi}
\begin{itemize}
\item[$i)$] Let $\sigma:(0,T)\times \mathbb{R}\rightarrow \mathbb{R}$ satisfy (S1)-(S3). For any $n\in\na$, $\lambda\in\re$, and a.e. $t\in (0,T)$ we introduce the Lipschitz regularization of $\sigma$ via inf-convolution:
\begin{align}\label{231124_01}
\sigma_n(t,\lambda):=\inf_{\mu\in\re}(\sigma(t,\mu)+n|\lambda-\mu|),
\end{align}
see, e.g., \cite[Theorem 9.2.1]{ABM06}.
\item[$ii)$] Let $B$ be defined by \eqref{defi_B}. Then we define for any $n\in\na$, $v,\varphi\in L^2(D)$, and a.e. $(t,x)\in(0,T)\times D$
\begin{align*}
B_n(t,v)\varphi(x):=\sigma_n(t,v(x))\int_D k(x,y)\varphi(y)\,dy.
\end{align*}
\end{itemize}
\end{defi}

\begin{prop}\label{231124_p1}
Let $\sigma:(0,T)\times\re\rightarrow\re$ satisfy (S1)-(S3), $C_{\sigma}$ be the constant given by $(S3)$,
\[n_0:=\left\lceil\sqrt{C_\sigma}\right\rceil=\min\{n\in\mathbb{N} : n\geq \sqrt{C_{\sigma}}\}.\] 
Then, the Lipschitz regularization of $\sigma$ via inf-convolution $\sigma_n$ has the following properties: There exists a full-measure set $U\subseteq (0,T)$ such that, for any $t\in U$
\begin{itemize}
\item[$i)$]$\sigma_n(t,\lambda)>-\infty$ for all $\lambda\in\mathbb{R}$,
\item[$ii)$] For all $\lambda\in\mathbb{R}$ and all $n\in\mathbb{N}$ such that $n\geq n_0$ 
\[\sigma_n(t,\lambda)\leq \sigma(t,\lambda)\]
\item[$iii)$] Lipschitz continuity:
 \[|\sigma_n(t,\lambda_1)-\sigma_n(t,\lambda_2)|\leq n|\lambda_1-\lambda_2|\]
 for all $n\in\na$ such that $n\geq n_0$ and all $\lambda_1,\lambda_2\in\re$.
 \item[$iv)$] Uniform boundedness with respect to $n\in \na$ and $\lambda\in\re$: There exists a constant $C_{\alpha}>0$, only depending on the Hölder exponent $\alpha\in (0,1)$ and the Hölder constant $L_{\alpha}>0$ of $\sigma$ such that
 \begin{align}\label{231124_03}
 |\sigma_n(t,\lambda)-\sigma(t,\lambda)|\leq C_{\alpha}
 \end{align}
 for all $n\in\mathbb{N}$ such that $n\geq n_0$ and all $\lambda\in\mathbb{R}$. Moreover,
 \begin{align}\label{231124_04}
 |\sigma_n(t,\lambda)|^2\leq 2(C_{\alpha}^2+C_{\sigma}(1+|\lambda|^2))
 \end{align}
for all $\lambda\in\re$. 
\item[$v)$] Uniform convergence:
\begin{align*}
\lim_{n\rightarrow \infty}\sup_{t\in U}\sup_{\lambda\in \re}|\sigma_n(t,\lambda)-\sigma(t,\lambda)|=0.
\end{align*}
\end{itemize}
\end{prop}

\begin{proof}
We choose the full-measure set $U\subseteq(0,T)$ such that (S1)-(S3) hold true for all $t\in U$.
$i)$ From (S1)-(S3) it follows that $\lambda\mapsto \sigma(t,\lambda)$ is continuous and 
\[\sigma(t,\lambda)\geq -\sqrt{C_{\sigma}(1+|\lambda|^2)}\geq -\sqrt{C_{\sigma}}(1+|\lambda|)\]
for all $t\in U$. Now, the result follows from \cite[Theorem 9.2.1]{ABM06}.\\
$ii)$ follows immediately by discarding the infimum and plugging $\mu=\lambda$ in \eqref{231124_01}.\\
$iii)$ For any $t\in U$, any $n\in\na$ such that $n\geq n_0$ and any $\lambda_1,\lambda_2\in\re$ we have
\begin{align*}
&\sigma_n(t,\lambda_1)-\sigma_n(t,\lambda_2)\\[1ex]
&\leq \inf_{\mu\in\re}(\sigma(t,\mu)+n|\lambda_1-\lambda_2|+n|\lambda_2-\mu|)-\inf_{\mu\in\re}(\sigma(t,\mu)+n|\lambda_2-\mu|)\\[1ex]
&=n|\lambda_1-\lambda_2|+\inf_{\mu\in\re}
\{\sigma(t,\mu)+n|\lambda_2-\mu|\}-\inf_{\mu\in\re}(\sigma(t,\mu)+n|\lambda_2-\mu|)\\[1ex]
&= n |\lambda_1-\lambda_2|.
\end{align*}
With the same argument we obtain
\begin{align*}
\sigma_n(\omega,t,\lambda_2)-\sigma_n(\omega,t,\lambda_1)\leq n|\lambda_1-\lambda_2|.
\end{align*}
$iv)$ For any $t\in U$, any $n\in\na$ such that $n\geq n_0$ and any $\lambda\in\re$ using $ii)$ and $(S2)$ we get
\begin{align}\label{231119_01}
\begin{aligned}
&|\sigma(t,\lambda)-\sigma_n(t,\lambda)|=\sigma(t,\lambda)-\sigma_n(t,\lambda)\\[1ex]
&=\sigma(t,\lambda)-\inf_{\mu\in\re}(\sigma(t,\mu)+n|\lambda-\mu|)\\[1ex]
&=\sigma(t,\lambda)+\sup_{\mu\in\re}(-\sigma(t,\mu)-n|\lambda-\mu|)\\[1ex]
&=\sup_{\mu\in\re}(\sigma(t,\lambda)-\sigma(t,\mu)-n|\lambda-\mu|)\\[1ex]
&\leq \sup_{\mu\in\re}(|\sigma(t,\lambda)-\sigma(t,\mu)|-n|\lambda-\mu|)\\[1ex]
&\leq\sup_{\mu\in\re}\{L_\alpha|\lambda-\mu|^\alpha-n|\lambda-\mu|\}\leq\max_{r\in[0,\infty)} h_n(r),
\end{aligned}
\end{align}
where $h_n:[0,\infty)\rightarrow\re,\ h_n(r):=L_\alpha r^\alpha-nr$. For any $n\in\na$, we have $h_n(0)=0$ and $h_n'(r)=0$ iff $r=r_0^n$, where
\begin{align*}
r_0^n:=\left(\frac{n}{L_\alpha\alpha}\right)^{\frac{1}{\alpha-1}}.
\end{align*}
Since $h_n'(r)>0$ for all $0<r<r_0^n$ and $h_n'(r)<0$ for all $r>r_0^n$ and
\[h_n(r_0^n)=\frac{n^{\frac{\alpha}{\alpha-1}}(1-\alpha)}{L_\alpha^{\frac{1}{\alpha-1}}\alpha^{\frac{\alpha}{\alpha-1}}}>0\]
for $\alpha\in (0,1)$, it follows that 
\begin{align}\label{231119_02}
\max_{r\in[0,\infty)} h_n(r)=h_n(r_0^n)=\frac{n^{\frac{\alpha}{\alpha-1}}(1-\alpha)}{L_\alpha^{\frac{1}{\alpha-1}}\alpha^{\frac{\alpha}{\alpha-1}}}.
\end{align}
Since $\frac{\alpha}{\alpha-1}<0$, we have 
\begin{align}\label{231119_03}
\frac{n^{\frac{\alpha}{\alpha-1}}(1-\alpha)}{L_\alpha^{\frac{1}{\alpha-1}}\alpha^{\frac{\alpha}{\alpha-1}}}\leq \frac{1-\alpha}{L_\alpha^{\frac{1}{\alpha-1}}\alpha^{\frac{\alpha}{\alpha-1}}}=:C_{\alpha}
\end{align}
for all $n\in\na$ and \eqref{231124_03} holds true. Now, using \eqref{231124_03} and $(S3)$, we have
\begin{align*}
|\sigma_n(t,\lambda)|^2&\leq 2(|\sigma_n(t,\lambda)-\sigma(t,\lambda)|^2+|\sigma(t,\lambda)|^2)\\
&\leq 2(C_{\alpha}+C_{\sigma}(1+|\lambda|^2)),
\end{align*}
hence \eqref{231124_04}.
$v)$ We recall that for all $t\in U$, all $n\geq n_0$ and all $\lambda\in\mathbb{R}$ \eqref{231119_02} holds true where $\frac{\alpha}{\alpha-1}<0$ for $\alpha\in (0,1)$. Hence,
\begin{align*}
\lim_{n\rightarrow \infty} \sup_{t\in U}\sup_{\lambda\in \re}|\sigma_n(t,\lambda)-\sigma(t,\lambda)|\leq \lim_{n\rightarrow\infty }\frac{n^{\frac{\alpha}{\alpha-1}}(1-\alpha)}{L_\alpha^{\frac{1}{\alpha-1}}\alpha^{\frac{\alpha}{\alpha-1}}}=0.
\end{align*}
\end{proof}

From Proposition \ref{231124_p1} we get the following consequences:

\begin{cor}\label{lem_Lipschitz_B_n}
Let $\sigma:(0,T)\times\re\rightarrow\re$ satisfy (S1)-(S3), $C_{\sigma}$ be the constant given by $(S3)$, and $n_0:=\lceil \sqrt{C_\sigma}\rceil$.
Then,
\begin{itemize}
\item[$i)$] For a.e. $t\in (0,T)$, the mapping $L^2(D)\ni v\mapsto  B_n(t,v)$ is Lipschitz continuous from $L^2(D)$ to $HS(L^2(D))$ with Lipschitz constant $L_{B_n}=\sqrt{C_k}n$ for all $n\geq n_0$.
\item[$ii)$] For any $u\in L^2(\Omega,L^2(0,T;L^2(D))$ we have
\[\lim_{n\rightarrow\infty}\erww{\int_0^T\Vert B_n(t,u)-B(t,u)\Vert_{HS}^2\,dt}=0.\]
\item[$iii)$] For all $n\geq n_0$, all $v\in L^2(D)$ and a.e. $t\in (0,T)$ we have 
\[\Vert B_n(t,v)\Vert^2_{HS}\leq 2[(C_{\alpha}^2+C_{\sigma})\|k\|_{L^2(D\times D)}^2+C_\sigma C_k\Vert v\Vert_2^2)].\]
\end{itemize}
\end{cor}

\begin{proof}
 $i)$ Recalling \eqref{estimateBsigma} and using Proposition~\ref{231124_p1}~$iii)$ for any $v,w\in L^2(D)$ and a.e. $t\in (0,T)$ we get
 \begin{align*}
\|B(t,v)-B(t,w)\|_{\HS}^2
&\leq \int_D|\sigma_n(t,v(x))-\sigma_n(t,w(x))|^{2}\int_D|k(x,y)|^2\,dy\,dx\\
&\leq C_k n^2\|v-w\|_{2}^{2}.
 \end{align*}
$ii)$ With similar arguments as in \eqref{B_bounded},  \eqref{estimateBsigma} and Proposition~\ref{231124_p1}~$v)$, we get for all $n\geq n_0$
\begin{align*}
    &\erww{\int_0^T\Vert B_n(t,u)-B(t,u)\Vert_{HS}^2\,dt}\\
    &\leq \erww{\int_0^T\int_D|\sigma_n(t,u(\omega,t,x)-\sigma(t,u(\omega,t,x))|^{2}\int_D|k(x,y)|^2\,dy\,dx}\\
&\leq C_k\erww{\int_0^T\int_D|\sigma_n(t,u(\omega,t,x)-\sigma(t,u(\omega,t,x))|^{2}\,dx\,dt}\leq C_kT|D| h_n^2(r_0)
\end{align*}
and the last term on the right-hand side is converging to $0$ for $n\rightarrow\infty$. 
$iii)$ is a direct consequence of \eqref{B_bounded} and of Proposition \ref{231124_p1}, \eqref{231124_04}.
\end{proof}

\subsection{A higher order perturbation}\label{Sec_ApprEqu}

Let $m\in\na$ be chosen such that
\begin{align}\label{embedding}
H^m_0(D)\hookrightarrow W^{1,2p}_0(D)\cap L^\infty(D).
\end{align}
For $q:=\max\{2,p,2p(p-1),p'\}$ we consider the Gelfand triple
\begin{align*}
W^{m,q}_0(D)\hookrightarrow L^2(D) \hookrightarrow W^{-m,q'}(D)
\end{align*}
and define for $n\in\na$ the operator $A_n:W^{m,q}_0(D)\rightarrow W^{-m,q'}(D)$ by
\begin{align*}
\langle A_n(u),v\rangle_{q',q}:=&\int_D a(x,u,\nabla u)\nabla v\,dx+\frac{1}{n} j(u,v)+\int_D f(u)v\,dx,\quad u,v\in\sobolevmq,
\end{align*}
where $\langle\cdot,\cdot\rangle_{q',q}$ denotes the duality bracket $\langle\cdot,\cdot\rangle_{W^{-m,q'}(D),W^{m,q}_0(D)}$ and
\begin{align*}
    j(u,v):=(u,v)_{H^m_0(D)}+\int_D\sum_{|\gamma|\leq m}|D^\gamma u|^{q-2}D^\gamma u\cdot D^\gamma v\,dx, \quad u,v\in\sobolevmq
\end{align*}
denotes the variational formulation of the maximal monotone operator associated to the Gâteaux derivative of $J:\sobolevmq\rightarrow \re$, $J(v):=\frac{1}{q}\|v\|_{\sobolevmq}^q+\halbe\|v\|_{H^m_0(D)}^2$.

For $n\in\na, n\geq n_0:=\lceil \sqrt{C_\sigma}\rceil$ we consider the approximated equation
\begin{align}\label{apprEqu}
\begin{aligned}
d u_n+A_n(u_n)\,dt&=B_n(t,u_n)\,dW_t &&\text{in }\Omega\times(0,T)\times D\\
u&=0 &&\text{on }\Omega\times(0,T)\times\partial D\\
u_n(0,\cdot)&=u_0 &&\text{in }\Omega\times D,
\end{aligned}
\end{align}

\subsection{Well-posedness of the approximated equation}

In the following we denote by $C_E$ constants arriving from embeddings and let $n_0:=\lceil \sqrt{C_\sigma}\rceil$.

\begin{lem}\label{lem_H2}
For fixed $n\in\na, n\geq n_0,$ there exists a constant $C_{\ref{lem_H2}}\in\re$ and a function $\rho:\sobolevmq\rightarrow[0,\infty)$, which is measurable, hemi-continuous, and locally bounded in $\sobolevmq$, such that for all $u,v\in\sobolevmq$, a.e. $t\in(0,T)$
\begin{align*}
-2\langle A_n(u)-A_n(v),u-v\ranglev+\|B_n(t,u)-B_n(t,v)\|_{\HS}^2\leq(C_{\ref{lem_H2}}+\rho(v))\|u-v\|_2^2.
\end{align*}
\end{lem}

\begin{proof}
Let $u,v\in\sobolevmq$ be arbitrary and $n\in\na,n\geq n_0,$ be fixed. We know for a.e. $t\in(0,T)$
\begin{align}\label{proofH2_1}
\begin{split}
&-2\langle A_n(u)-A_n(v),u-v\ranglev+\|B_n(t,u)-B_n(t,v)\|_\HS^2\\
&=-2\int_D(a(x,u,\nabla u)-a(x,v,\nabla v))\nabla(u-v)\,dx-\frac{2}{n}\|u-v\|_{H^m_0}^2\\
&\quad-\frac{2}{n} \int_D\sum_{|\gamma|\leq m}(|D^\gamma u|^{q-2}D^\gamma u-|D^\gamma v|^{q-2}D^\gamma v)D^\gamma (u-v)\,dx\\
&\quad-2\int_D(f(u)-f(v))(u-v)\,dx+\|B_n(t,u)-B_n(t,v)\|_\HS^2.
\end{split}
\end{align}
By using (A1), (A3), and Hölder's inequality we obtain
\begin{align*}
&-2\int_D(a(x,u,\nabla u)-a(x,v,\nabla v))\nabla(u-v)\,dx\\
&=-2\int_D(a(x,u,\nabla u)-a(x,u,\nabla v))\nabla(u-v)\,dx\\
&-2\int_D(a(x,u,\nabla v)-a(x,v,\nabla v))\nabla(u-v)\,dx\\
&\leq 2\int_D(C_5|\nabla v|^{p-1}+h(x))|u-v||\nabla(u-v)|\,dx\\
&\leq 2^{p'+1}\|C_5|\nabla u|^{p-1}+h(x)\|_{p'}\|(u-v)\nabla(u-v)\|_p\\
&\leq 2^{p'+1}(C_5\|\nabla v\|_p^{p-1}+\|h\|_{p'})\|u-v\|_{2p}\|\nabla(u-v)\|_{2p}.
\end{align*}
Note that there holds $\|u-v\|_{2p}^{2p}\leq\|u-v\|_\infty^{2p-2}\|u-v\|_2^2$.
Thanks to the continuous embedding \eqref{embedding} and Young's inequality we have for $\eta>0$
\begin{align*}
&-2\int_D(a(x,u,\nabla u)-a(x,v,\nabla v))\nabla(u-v)\,dx\\
&\leq 2^{p'+1}(C_5\|\nabla v\|_p^{p-1}+\|h\|_{p'})\|u-v\|_\infty^{\frac{2p-2}{2p}}\|u-v\|_2^{\frac{1}{p}}C_E\|u-v\|_{H^m_0}\\
&\leq 2^{p'+1}(C_5\|\nabla v\|_p^{p-1}+\|h\|_{p'})C_E^{\frac{1}{p'}}\|u-v\|_{H_0^m}^{\frac{1}{p'}}\|u-v\|_2^{\frac{1}{p}}C_E\|u-v\|_{H^m_0}\\
&=K_1(\|\nabla v\|_p^{p-1}+1)\|u-v\|_2^{\frac{1}{p}}\|u-v\|_{H^m_0}^{1+\frac{1}{p'}}\\
&\leq \frac{1}{2p\eta}\left(K_1^{2p}(\|\nabla v\|_p^{p-1}+1)^{2p}\|u-v\|_2^2\right)+\frac{\eta}{(2p)'}\|u-v\|_{H_0^m}^2
\end{align*}
for a constant $K_1\geq0$ not depending on $n$. This estimate implies
\begin{align}\label{proofH2_2}
\begin{split}
&-2\int_D(a(x,u,\nabla u)-a(x,v,\nabla v))\nabla(u-v)\,dx-\frac{2}{n} \|u-v\|_{H^m_0}^2\\
&\leq \frac{2^{2p-1}}{\eta p}K_1^{2p}(1+\|\nabla v\|_p^{2p(p-1)})\|u-v\|_2^2+\left(\frac{2p-1}{2p}\eta-\frac{2}{n} \right)\|u-v\|_{H_0^m}^2.
\end{split}
\end{align}
In the following we choose $\eta>0$ small enough such that $\frac{2p-1}{2p}\eta-\frac{2}{n} <0$. Note that
\begin{align*}
\int_D\sum_{|\gamma|\leq m}(|D^\gamma u|^{q-2}D^\gamma u-|D^\gamma v|^{q-2}D^\gamma v)D^\gamma (u-v)\,dx\geq0.
\end{align*}
Therefore we obtain for a.e. $t\in(0,T)$ from \eqref{proofH2_1} by using \eqref{proofH2_2}, Corollary \ref{lem_Lipschitz_B_n} $iii)$, and the Lipschitz continuity of $f$
\begin{align*}
&-2\langle A_n(u)-A_n(v),u-v\ranglev+\|B_n(t,u)-B_n(t,v)\|_\HS^2\leq (C_{\ref{lem_H2}}+K_2\|\nabla v\|_{p}^{2p-2})\|u-v\|_2^2
\end{align*}
for constants $C_{\ref{lem_H2}}\in\re,K_2\geq 0$ both depending on $n$. Hence for $$\rho(v):=K_2\|\nabla v\|_p^{2p-2}$$ the assertion is satisfied. Note that, by using $q\geq 2p(p-1)$ and the embedding $W_0^{m,q}(D)\hookrightarrow W_0^{1,p}(D)$ we have for any $v\in\sobolevmq$
\begin{align}\label{wachstumsbed_rho}
\begin{split}
\rho(v)&=K_2\|\nabla v\|_p^{2p-2}\leq K_2 2^q(1+\|\nabla v\|_{p}^q)\leq K_2 2^q(1+C_E^q\|v\|_{W_0^{m,q}}^q).
\end{split}
\end{align}
\end{proof}

\begin{lem}\label{lem_H3}
For $n\in\na$ large enough there exist constants $C_{\ref{lem_H3}}^1,C_{\ref{lem_H3}}^2\in\re$, and $\theta\in(0,\infty)$ such that for all $u\in\sobolevmq$ and a.e. $t\in(0,T)$
\begin{align*}
-2\langle A_n(u),u\ranglev +\|B_n(t,u)\|_{\HS}^2\leq C_{\ref{lem_H3}}^1\|u\|_2^2-\theta\|u\|_{W_0^{m,q}}^q+C_{\ref{lem_H3}}^2.
\end{align*}
\end{lem}

\begin{proof}
Let $n\in\na,n\geq n_0$ be arbitrary, but fixed. By (A2), the Lipschitz continuity of $f$, and Corollary \ref{lem_Lipschitz_B_n} $iii)$ we obtain for all $u\in\sobolevmq$ and a.e. $t\in(0,T)$
\begin{align}\label{proofH3_1}
\begin{split}
&-2\int_D a(x,u,\nabla u)\nabla u\,dx-\frac{2}{n} \|u\|_{H_0^m}^2-\frac{2}{n} \|u\|_{W_0^{m,q}}^q  - 2\int_D f(u)u\,dx+\|B_n(t,u)\|_{\HS}^2\\
&\leq -2\int_D(\kappa(x)+C_1|\nabla u|^p-C_2|u|^\nu)\,dx-\frac{2}{n} \|u\|_{W_0^{m,q}}^q+2L_f\|u\|_2^2\\
&\quad +2(C_\alpha^2+C_\sigma)\knorm^2+ 2C_kC_\sigma\|u\|_2^2\\
&\leq 2\|\kappa\|_{1}+C_2\int_D |u|^\nu\,dx-\frac{2}{n} \|u\|_{W_0^{m,q}}^q+2(L_f+C_\sigma C_k)\|u\|_{2}^2+2(C_\alpha^2+C_\sigma)\|k\|_{L^2(D\times D)}^2.
\end{split}
\end{align}
By using the continuous embeddings $L^\nu(D)\hookrightarrow L^{2p}(D)$ and $\sobolevmq\hookrightarrow W_0^{1,2p}(D)$, Poincaré's inequality (with constant $C_{\text{Poin}}$), and the fact that $q\geq p>\nu$ we obtain
\begin{align*}
\int_D|u(x)|^\nu\,dx&\leq C_E^\nu\|u\|_{2p}^\nu\leq C_E^\nu C_{\text{Poin}}^\nu\|\nabla u\|_{2p}^\nu\leq C_E^{2\nu}C_{\text{Poin}}^\nu\|u\|_{W_0^{m,q}}^\nu\\
&\leq C_E^{2\nu}C_{\text{Poin}}^\nu 2^q\left(1+\|u\|_{W_0^{m,q}}^q\right).
\end{align*}
Consequently, we get from \eqref{proofH3_1} for all $u\in\sobolevmq$, a.e. $t\in(0,T)$, and for $n\in\na$ large enough such that $\theta:=2^qC_2C_E^{2\nu}C_{\text{Poin}}^\nu-\frac{2}{n} >0$
\begin{align*}
-2\langle A_n(u),u\ranglev +\|B_n(t,u)\|_{\HS}^2\leq C_{\ref{lem_H3}}^1\|u\|_2^2-\theta\|u\|_{W_0^{m,q}}^q+C_{\ref{lem_H3}}^2
\end{align*}
for constants $C_{\ref{lem_H3}}^1,C_{\ref{lem_H3}}^2\geq0$.
\end{proof}

\begin{lem}\label{lem_H4}
For fixed $n\in\na,n\geq n_0,$ there exist $C_{\ref{lem_H4}}^1,C_{\ref{lem_H4}}^2\in\re$ such that
\begin{align*}
\|A_n(u)\|_{W^{-m,q'}(D)}^{q'}\leq C_{\ref{lem_H4}}^1+C_{\ref{lem_H4}}^2\| u\|_{W_0^{m,q}}^q\quad\forall u\in\sobolevmq.
\end{align*}
\end{lem}

\begin{proof}
Let $n\in\na,n\geq n_0,$ be fixed. For $u,v\in\sobolevmq$ we obtain by using Hölder's inequality, (A2), and the Lipschitz continuity of $f$
\begin{align*}
\langle A_n(u),v\ranglev&\leq \|a(x,u,\nabla u)\|_{p'}\|\nabla v\|_p+\frac{1}{n} \|u\|_{H^m_0}\|v\|_{H_0^m}+\frac{1}{n} \|u\|_{W_0^{m,q}}^{q-1}\|v\|_{W^{m,q}_0}+ L_f\|u\|_2\|v\|_2\\
&\leq 2(C_3\|\nabla u\|_p^{p-1}+C_4\|u\|_p^{p-1}+\|g\|_{p'})\|\nabla v\|_p+\frac{1}{n} \|u\|_{H^m_0}\|v\|_{H_0^m}\\
&\quad+\frac{1}{n} \|u\|_{W_0^{m,q}}^{q-1}\|v\|_{W^{m,q}_0}+ L_f\|u\|_2\|v\|_2.
\end{align*}
Note that $W^{m,q}_0(D)$ is compactly embedded into $W^{1,p}_0(D)\cap H^m_0(D)$, which implies with Poincaré's inequality for any $u,v\in\sobolevmq$
\begin{align*}
\langle A_n(u),v\ranglev&\leq 2\left[C_E(C_3+C_{\text{Poin}}^{p-1}C_4)\| u\|_{W_0^{m,q}}^{p-1}+\|g\|_{p'}\right]C_E\| v\|_{W^{m,q}_0}\\
&\quad+\frac{C_E^2}{n}\|u\|_{W^{m,q}_0}\|v\|_{W_0^{m,q}}+\frac{1}{n} \|u\|_{W_0^{m,q}}^{q-1}\|v\|_{W^{m,q}_0}+ L_fC_E^2\|u\|_{W_0^{m,q}}\|v\|_{W_0^{m,q}}.
\end{align*}
Consequently there exist constants $K_3,K_4\geq0$ depending on $n$ such that
\begin{align*}
\|A_n(u)\|_{W^{-m,q'}}^{q'}\leq K_3(\|u\|_{W_0^{m,q}}^{q'(p-1)}+\|u\|_{W_0^{m,q}}^{q'}+\|u\|_{W_0^{m,q}}^q)+K_4.
\end{align*}
By definition of $q$ we know that $q'<q$ and $q'(p-1)\leq q$ and hence
\begin{align*}
\|A_n(u)\|_{W^{-m,q'}}^{q'}\leq C_{\ref{lem_H4}}^1+C_{\ref{lem_H4}}^2\|u\|_{W_0^{m,q}}^q\quad\forall u\in\sobolevmq
\end{align*}
for constants $C_{\ref{lem_H4}}^1,C_{\ref{lem_H4}}^2\in\re$ depending on $n$.
\end{proof}\\

In the following, let
\begin{align}\label{def_N0} 
N_0:=\max\left\{\lceil \sqrt{C_\sigma}\rceil, \left(2^{q-1}C_2C_E^{2\nu}C_{\text{Poin}}^\nu\right)^{-1}\right\},
\end{align}
i.e., if we choose $n>N_0$, then $n$ is large enough such that Lemma~\ref{lem_H3} holds.

\begin{prop}\label{teo_existence_appr_solution}
For any $n\in\na,n>N_0,$ there exists a unique solution $u_n$ to the approximated equation \eqref{apprEqu} in the following sense: $u_n \in L^2(\Omega;C([0,T];L^2(D)))\cap L^q(\Omega;\linebreak L^q(0,T;W^{m,q}_0(D)))$ is a $(\mathcal{F}_t)_{t\in[0,T]}$-adapted stochastic process which satisfies $u_n(0,\cdot)=u_0$ in $\Omega\times D$ and for all $t\in[0,T]$, $\mP$-a.s. in $\Omega$
\begin{align*}
    u_n(t)=u_0+\int_0^t A_n(u_n(s))\,ds+\int_0^t B_n(s,u_n(s))\,dW_s.
\end{align*}
\end{prop}

\begin{proof}
Using Lemma \ref{lem_H2}, Lemma \ref{lem_H3}, and Lemma \ref{lem_H4} in connection with \eqref{wachstumsbed_rho} the result follows from \cite[Theorem 5.1.3]{LR15}.
\end{proof}
\\

Since the case $\alpha=1$ is already known (see, e.g., \cite{LR17},\cite{P21}), we only consider $\alpha\in(0,1)$ in the following.

\section{Existence of a martingale solution}

\subsection{\textit{A priori} estimates}\label{Sec_A_priori}

In the following for $n\in\na,n>N_0,$ let $u_n$ be the solution function to \eqref{apprEqu} found in Proposition~\ref{teo_existence_appr_solution}, where $N_0$ is defined in \eqref{def_N0}. 

\begin{lem}\label{lemboundungradun}
There exists a constant $C_{\ref{lemboundungradun}}\geq0$ not depending on $n\in\na$ such that for all $n\in\na,n>N_0,$ and $t\in[0,T]$
\begin{align*}
&\erww{\|u_n(t)\|_2^2}+\erww{\int_0^t\|\nabla u_n(s)\|_p^p\,ds}+\frac{1}{n} \erww{\int_0^t\|u_n(s)\|_{H_0^m}^2\,ds}\\
&+\frac{1}{n} \erww{\int_0^t\|u_n(s)\|_{W_0^{m,q}}^q\,ds}\leq C_{\ref{lemboundungradun}}.
\end{align*}
\end{lem}

\begin{proof}
Let $n\in\na,n>N_0,$ and $t\in[0,T]$ be arbitrary. We get from Itô's formula, $\mP$-a.s. in $\Omega$,
\begin{align*}
\|u_n(t)\|_2^2&=\|u_0\|_2^2+\int_0^t2\langle\diver a(\cdot,u_n,\nabla u_n),u_n(s)\ranglev\,ds-\frac{2}{n} \int_0^t\|u_n(s)\|_{H_0^m}^2\,ds\\
&\quad-\frac{2}{n} \int_0^t\int_D\sum_{|\gamma|\leq m}|D^\gamma u_n(s)|^q\,dx\,ds-\int_0^t\int_D 2 f(u_n(s))u_n(s)\,dx\,ds\\
&\quad+\int_0^t\|B_n(s,u_n(s))\|_{\HS}^2\, ds+2\int_0^t\langle B_n(s,u_n(s))(\cdot),u_n(s)\rangle_{L^2}\,dW_s.
\end{align*}
Taking the expectation provides by using the Lipschitz continuity of $f$ and\linebreak Corollary~\ref{lem_Lipschitz_B_n}~$iii)$
\begin{align}\label{proof_apriori_1}
\begin{split}
&\erww{\|u_n(t)\|_2^2}+2\erww{\int_0^t\int_D a(x,u_n,\nabla u_n)\nabla u_n(s)\,dx\,ds}+\frac{2}{n} \erww{\int_0^t\|u_n(s)\|_{H^m_0}^2\,ds}\\
&+\frac{2}{n} \erww{\int_0^t\|u_n(s)\|_{W_0^{m,q}}^q\,ds}\\
&\leq\erww{\|u_0\|_2^2}+2L_f\erww{\int_0^t\|u_n(s)\|_2^2\,ds}\\
&\quad +\erww{\int_0^t 2\left((C_\alpha^2+C_\sigma)\|k\|_{L^2(D\times D)}^2+C_\sigma C_k\|u_n(s)\|_2^2\right)\,ds}\\
&=\erww{\|u_0\|_2^2}+2(C_\alpha^2+C_\sigma) t\|k\|_{L^2(D\times D)}^2+2(C_\sigma C_k+L_f)\erww{\int_0^t\|u_n(s)\|_2^2\,ds}.
\end{split}
\end{align}
By using (A2) we obtain
\begin{align*}
&\erww{\|u_n(t)\|_2^2}+2\erww{\int_0^t\int_D\kappa(x)+C_1|\nabla u_n(s,x)|^p-C_2|u_n(s,x)|^\nu\,dx\,ds}\\
&\leq \erww{\|u_0\|_2^2}+2(C_\alpha^2+C_\sigma) t\knorm^2+2(C_\sigma C_k+L_f)\erww{\int_0^t\|u_n(s)\|_2^2\,ds}.
\end{align*}
Let $\eta>0$, since $\nu<p$ we can use Young's inequality in the following way:
\begin{align*}
\erww{\int_0^t\int_D|u_n(s,x)|^\nu\,dx\,ds}\leq\eta\frac{\nu}{p}\erww{\int_0^t\int_D|u_n(t,x)|^p\,dx\,ds}+\frac{1}{\eta}\frac{p}{p-\nu}t|D|.
\end{align*}
Applying Poincaré's inequality we obtain by \eqref{proof_apriori_1}
\begin{align*}
&\erww{\|u_n(t)\|_2^2}+2\left(C_1-C_2C_{\text{Poin}}^p\frac{\nu}{p}\eta\right)\erww{\int_0^t\|\nabla u_n(s)\|_p^p\,ds}\\
&\leq \erww{\|u_0\|_2^2}+\left(2(C_\alpha^2+C_\sigma)\knorm^2+2\|\kappa\|_1+2C_2\frac{p}{\eta(p-\nu)}|D|\right)T\\
&\quad+2(C_kC_\sigma+L_f)\erww{\int_0^t\|u_n(s)\|_2^2},
\end{align*}
where we can choose $\eta>0$ small enough such that $C_1-2C_2C_{\text{Poin}}^p\frac{\nu}{p}\eta>0$.
Using Gronwall's lemma we get for all $t\in[0,T]$ and $n>N_0$
\begin{align*}
\erww{\|u_n(t)\|_2^2}&\leq K_5(1+K_6Te^{K_6T})
\end{align*}
for constants $K_5,K_6\geq0$ independent of $n$ and $t$.
Consequently, we get from \eqref{proof_apriori_1} by using (A2)
\begin{align*}
&\erww{\|u_n(t)\|_2^2}+2\left(C_1-C_2C_{\text{Poin}}^p\frac{\nu}{p}\eta\right)\erww{\int_0^t\|\nabla u_n(s)\|_p^p\,ds}\\
&+\frac{2}{n} \erww{\int_0^t\|u(s)\|_{H^m_0}^2\,ds}+\frac{2}{n} \erww{\int_0^t\|u_n(s)\|_{W_0^{m,q}}^q\,ds}\\
&\leq\erww{\|u_0\|_2^2}+2\|\kappa\|_1+2(C_\alpha^2+C_\sigma) t\|k\|_{L^2(D\times D)}+2(C_\sigma C_k+L_f)TK_5\left(1+K_6Te^{K_6T}\right)
\end{align*}
for all $t\in[0,T]$ and $n>N_0$.
\end{proof}

\begin{lem}\label{lem_sup_u_n_bound}
There exists a constant $C_{\ref{lem_sup_u_n_bound}}\geq0$ not depending on $n\in\na$ such that
\begin{align*}
\erww{\sup_{t\in[0,T]}\|u_n(t)\|_2^2}\leq C_{\ref{lem_sup_u_n_bound}}\quad\forall n>N_0.
\end{align*}
\end{lem}

\begin{proof}
Let $n\in\na,n>N_0,$ be arbitrary. Using Itô's formula, (A2), and Corollary~\ref{lem_Lipschitz_B_n}~$iii)$ we get for all $t\in[0,T]$, $\mP$-a.s. in $\Omega$,
\begin{align*}
\|u_n(t)\|_2^2&\leq\|u_0\|_2^2-2\int_0^t\int_D a(x,u_n,\nabla u_n)\nabla u_n(s)\,dx\,ds\\
&\quad-2\int_0^t\int_Df(u_n(s))u_n(s)\,dx\,ds+\int_0^t\|B_n(s,u_n(s))\|_{\HS}^2\,ds\\
&\quad+2\int_0^t\langle B_n(s,u_n(s))(\cdot),u_n(s)\rangle_{L^2}\,dW_s\\
&\leq \|u_0\|_2^2+2\|\kappa\|_1+C_2\int_0^t\int_D|u_n(s,x)|^\nu\,dx\,ds+2(L_f+C_\sigma C_k)\int_0^t \|u_n(s)\|_2^2\,ds\\
&\quad +2(C_\alpha^2+C_\sigma)\|k\|_{L^2(D\times D)}^2t+2\int_0^t(B_n\left(s,u_n(s))(\cdot),u_n(s)\right)_{L^2}\,dW_s.
\end{align*}
Since $\nu<p$ and therefore $L^p(D)\hookrightarrow L^\nu(D)$, we obtain by applying Poincaré's inequality (with constant $C_{\text{Poin}}$) for all $t\in[0,T]$, $\mP$-a.s. in $\Omega$,
\begin{align*}
\int_0^t\|u_n(s)\|_\nu^\nu\,ds&\leq C_E^\nu\int_0^t\|u_n(s)\|_p^\nu\leq C_E^\nu C_{\text{Poin}}^\nu\int_0^t\|\nabla u_n(s)\|_p^\nu\,ds\\
&\leq C_E^\nu C_{\text{Poin}}^\nu 2^p\int_0^t\left(1+\|\nabla u_n(s)\|_p^p\right)\,ds.
\end{align*}
Taking first the supremum over all $t\in[0,T]$ and then the expectation provides
\begin{align}\label{eqhelpsupun}
\begin{aligned}
&\erww{\sup_{t\in[0,T]}\|u_n(t)\|_2^2}\\
&\leq\erww{\|u_0\|_2^2}+2\|\kappa\|_1T+C_2C_E^\nu C_{\text{Poin}}^\nu\erww{\int_0^T\|\nabla u_n(s)\|_p^p\,ds}\\
&\quad+2(L_f+C_\sigma C_k)\erww{\int_0^T\|u_n(s)\|_2^2\,ds}+(2(C_\alpha^2+C_\sigma)\|k\|_{L^2(D\times D)}^2+C_2C_E^\nu C_{\text{Poin}}^\nu)T\\
&\quad+2\erww{\sup_{t\in[0,T]}\left|\int_0^t\langle B_n(s,u_n(s))(\cdot),u_n(s)\rangle_{L^2}\,dW_s\right|}.
\end{aligned}
\end{align}
Using the Burkholder-Davis-Gundy inequality (see \cite[Thm. 1.1.7]{LR17}) and Young's inequality with $\beta>0$ such that $1-2C_{BDG}\beta>0$ we get
\begin{align*}
&\erww{\sup_{t\in[0,T]}\left|\int_0^t\langle B_n(s,u_n(s))(\cdot),u_n(s)\rangle_{L^2}\,dW_s\right|}\\
&\leq C_{BDG}\erww{\left(\int_0^T|\langle B_n(s,u_n(s))(\cdot),u_n(s)\rangle_{L^2}|^2 ds\right)^\halbe}\\
&\leq C_{BGD}\erww{\left(\sup_{t\in[0,T]}\|u_n(t)\|_2^2\int_0^T\|B_n(s,u_n(s))\|_{\HS}^2\,ds\right)^\halbe}\\
&\leq C_{BDG}\erww{\beta\supt\|u_n(t)\|_2^2}+C_{BDG}\erww{\beta^{-1}\int_0^T\|B_n(s,u_n(s))\|_\HS^2\,ds}.
\end{align*}
Using this inequality in \eqref{eqhelpsupun} implies
\begin{align*}
&(1-2C_{BDG}\beta)\erww{\supt\|u_n(t)\|_2^2}\\
&\leq C_2C_E^\nu C_{\text{Poin}}^\nu \erww{\int_0^T\|\nabla u_n(s)\|_p^p\,ds} + 2(L_f+C_\sigma C_k)\erww{\int_0^T\|u_n(s)\|_2^2\,ds}\\
&\quad+2C_{BDG}\beta^{-1}\erww{\int_0^T\|B_n(s,u_n(s))\|_\HS^2\,ds}+K_7
\end{align*}
for a constant $K_7\geq0$ not depending on $n$ and $t$. The assertion follows from Lemma~\ref{lemboundungradun} and Corollary~\ref{lem_Lipschitz_B_n} $iii)$.
\end{proof}

\begin{lem}\label{a(u_n)_bounded}
The sequence
\begin{align*}
(a(\cdot,u_n,\nabla u_n))_{n>N_0} \text{ is bounded in }L^{p'}(\Omega;L^{p'}(0,T;L^{p'}(D))).
\end{align*}
\end{lem}

\begin{proof}
The boundedness follows from (A2), Poincarés inequality and Lemma \ref{lemboundungradun}.
\end{proof}

\begin{lem}\label{lem_bound_derivative_u_n-B_n}
The sequence
\begin{align*}
\left(\partial_t\left(u_n(t)-\int_0^tB_n(s,u_n(s))dW_s\right)\right)_{n>N_0}
\end{align*}
is bounded in $L^{q'}(\Omega;L^{q'}(0,T;W^{-m,q'}(D)))$.
\end{lem}

\begin{proof}
Let $n\in\na,n>N_0,$ be arbitrary. Since $u_n\in L^2(\Omega;C([0,T];L^2(D)))\cap L^q(\Omega;\linebreak L^q(0,T;W_0^{m,q}(D)))$ is a solution to \eqref{apprEqu} by Proposition~\ref{teo_existence_appr_solution} we get for $t\in[0,T],\varphi\in\sobolevmq$, $\mP$-a.s. in $\Omega$
\begin{align*}
&\langle \partial_t\left( u_n(t)-\int_0^tB_n(s,u_n(s))\,dW_s\right),\varphi\ranglev\\
&=-\int_D a(x,u_n(t),\nabla u_n(t))\nabla \varphi\,dx-\frac{1}{n} (u_n(t),\varphi)_{H_0^m}\\
&\quad-\frac{1}{n} \int_D\sum_{|\gamma|\leq m}|D^\gamma u_n(t)|^{q-2}D^\gamma u_n(t)\cdot D^\gamma\varphi\,dx 
-\int_Df (u_n(t))\varphi\,dx\\
&\leq\|a(\cdot,u_n(t),\nabla u_n(t))\|_{p'}\|\nabla\varphi\|_p+\frac{1}{n} \|u_n(t)\|_{H_0^m}\|\varphi\|_{H_0^m}+\frac{1}{n} \|u_n(t)\|_{W_0^{m,q}}^{q-1}\|\varphi\|_{W_0^{m,q}}\\
&\quad+L_f\|u_n(t)\|_2\|\varphi\|_2.
\end{align*}
Since we have the continuous embedding $\sobolevmq\hookrightarrow W_0^{1,p}(D)\cap H_0^m(D)\cap L^2(D)$, we know that there exists a constant $C_E\geq0$ such that
\begin{align*}
(\|\varphi\|_{H_0^m}+\|\nabla\varphi\|_p+\|\varphi\|_2)\leq C_E\|\varphi\|_{W^{m,q}_0}.
\end{align*}
Hence we obtain by taking the supremum over all $\varphi\in\sobolevmq$ with $\|\varphi\|_{\sobolevmq}=1$
\begin{align*}
&\left\|\partial_t\left(u_n(t)-\int_0^tB_n(s,u_n(s))\,dW_s\right)\right\|_{W^{-m,q'}}\\
&\leq C_E\|a(x,u_n(t),\nabla u_n(t))\|_{p'}+\frac{C_E}{n} \|u_n(t)\|_{H_0^m}+\frac{1}{n} \|u_n(t)\|_{W^{m,q}_0}^{q-1}+C_EL_f\|u_n(t)\|_2\\
&\leq C_E\|a(x,u_n(t),\nabla u_n(t))\|_{p'}+\frac{C_E^2}{n}\|u_n(t)\|_{W_0^{m,q}}+\frac{1}{n} \|u_n(t)\|_{W_0^{m,q}}^{q-1}+C_E L_f\|u_n(t)\|_2\\
&\leq C_E\|a(x,u_n(t),\nabla u_n(t))\|_{p'} +2^{q-1}\frac{C_E^2}{n} +\frac{1}{n} (2^{q-1}C_E^2+1)\|u_n(t)\|_{W_0^{m,q}}^{q-1}+C_EL_f\|u_n(t)\|_2
\end{align*}
for all $t\in[0,T]$, $\mP$-a.s. in $\Omega$, where we used the fact that $q\geq2$. Because $q'\leq \min\{p',2\}$, there holds for all $t\in[0,T]$, $\mP$-a.s. in $\Omega$,
\begin{align*}
&\left\|\partial_t\left(u_n(t)-\int_0^tB_n(s,u_n(s))\,dW_s\right)\right\|_{W^{-m,q'}}^{q'}\\
&\leq 8^{q'}\Big(C_E^{q'}\|a(x,u_n(t),\nabla u_n(t)\|_{p'}^{q'}+2^q\left(\frac{C_E^2}{n}\right)^{q'}\\
&\quad+\frac{1}{n^{q'}}(2^{q-1}C_E^2+1)^{q'}\|u_n(t)\|_{W_0^{m,q}}^{q}+C_E^{q'}L_f^{q'}\|u_n(t)\|_2^{q'} \Big)\\
&\leq 8^{q'}\Big(2^{q'}C_E^{q'}(1+\|a(x,u_n(t),\nabla u_n(t)\|_{p'}^{p'})+2^qC_E^{2q'}\\
&\quad+\frac{1}{n} (2^{q-1}C_E^2+1)^{q'}\|u_n(t)\|_{W_0^{m,q}}^q+2C_E^{q'}L_f^{q'}(1+\|u_n(t)\|_2^2)\Big).
\end{align*}
Integrating on $[0,T]$ and taking the expectation provide the boundedness by Lemma~\ref{lemboundungradun} and Lemma~\ref{a(u_n)_bounded}.
\end{proof}

\begin{lem}\label{lem_bound_intB_n}
The sequence $(B_n(\cdot,u_n))_{n>N_0}$ is bounded in $L^2(\Omega;L^2(0,T;\HS(L^2(D))))$ and
$$\left(\int_0^\cdot B_n(s,u_n(s))\,dW_s\right)_{n>N_0}\text{ is bounded in }L^2(\Omega;C([0,T];\lz)),$$
both for a constant $C_{\ref{lem_bound_intB_n}}\geq0$ not depending on $n$.
\end{lem}

\begin{proof}
By Burkholder-Davis-Gundy inequality and Corollary \ref{lem_Lipschitz_B_n} $iii)$ we obtain for any $n\in\na,n>N_0,$
\begin{align*}
&\erww{\supt\left\|\int_0^t B_n(s,u_n(s))dW_s\right\|_2^2}\\
&\leq C_{BDG}\erww{\int_0^T\|B_n(s,u_n(s))\|_\HS^2\, ds}\\
&\leq C_{BDG}2\left((C_\alpha^2+C_\sigma)\knorm^2 T+C_\sigma C_k\erww{\int_0^T\|u_n(s)\|_2^2\,ds}\right)
\end{align*}
and by Lemma \ref{lemboundungradun} this expression is bounded.
\end{proof}

\subsection{Tightness results}\label{Sec_Tightness}

\begin{lem}\label{lem_undelta-intBn_tight}
The sequence
$$\left(u_n-\int_0^\cdot B_n(s,u_n(s))\,dW_s\right)_{n>N_0}$$
is bounded in $L^{q'}(\Omega;W^{\beta,2}(0,T;W^{-m,q'}(D)))$ for all $\beta\in(0,\halbe)$.
\end{lem}

\begin{proof}
We know that 
\begin{align*}
\mathfrak{V}:=\{v\in L^2(0,T;\lz):\partial_t v\in L^{q'}(0,T;W^{-m,q'}(D))\}
\end{align*}
is compactly embedded into $W^{1,q'}(0,T;W^{-m,q'}(D))$ and by \cite[Corollary 19]{S90} also in $W^{\beta,2}(0,T;W^{-m,q'}(D))$ for any $\beta\in(0,\halbe)$.
It follows from Lemma \ref{lem_sup_u_n_bound}, \ref{lem_bound_derivative_u_n-B_n}, \ref{lem_bound_intB_n} that 
\begin{align}\label{u_n-B_n_bounded}
\left(u_n-\int_0^\cdot B_n(s,u_n(s))\,dW_s\right)_{n>N_0}\text{ is bounded in }L^{q'}(\Omega;\mathfrak{V})
\end{align}
and therefore also in $L^{q'}(\Omega;W^{\beta,2}(0,T;W^{-m,q'}(D)))$. The assertion follows from a standard Markov inequality.
\end{proof}

\begin{lem}\label{lem_intBn_tight}
The sequence
$$\left(\int_0^\cdot B_n(s,u_n(s))\,dW_s\right)_{n>N_0}\text{ is bounded in }L^2(\Omega;W^{\beta,2}(0,T;\lz)).$$
\end{lem}

\begin{proof}
From Lemma \ref{lem_bound_intB_n} we know that $(B_n(\cdot,u_n))_{n>N_0}$ is bounded in $L^2(\Omega;L^2(0,T;\linebreak \HS(\lz)))$. Using \cite[Lemma 2.1, p.369]{FG95} we get for any $\beta\in(0,\halbe)$ and $n\in\na,n>N_0,$
\begin{align}\label{B_n_boundedinfraksobolev}
\begin{split}
\erww{\left\|\int_0^\cdot B_n(s,u_n(s))\,dW_s\right\|_{W^{\beta,2}(0,T;\lz)}^2}&\leq C(\beta)\erww{\int_0^T\|B_n(s,u_n(s))\|_{\HS}^2\,dt}\\
&\leq C(\beta)C_{\ref{lem_bound_intB_n}}.
\end{split}
\end{align}
\end{proof}

\begin{lem}\label{Lem_K_R_relativelycompact}
For all $R>0$ and $1\leq s<\infty$
\begin{align*}
K_R:=\left\{v\in L^p(0,T;\sobolevp)\cap W^{\beta,2}(0,T;W^{-m,q'}(D))\cap C([0,T];L^2(D)):\||v\||<R\right\}
\end{align*}
is relatively compact in $L^s(0,T;\lz)$, where
\begin{align*}
    \|| v\||:=\|v\|_{L^p(0,T;W^{1,p}_0(D))}+\|v\|_{W^{\beta,2}(0,T;W^{-m,q'}(D))}+\|v\|_{C([0,T];L^2(D))}.
\end{align*}
\end{lem}

\begin{proof}
Using the compact embeddings $\sobolevp\hookrightarrow \lz \hookrightarrow W^{-m,q'}(D)$ we obtain from \cite[Corollary 7]{S87} that $K_R$ is relatively compact in $L^s(0,T;\lz)$ for all $1\leq s<\infty$.
\end{proof}

\begin{lem}\label{lem_undelta_tight}
The sequence of laws of $(u_n)_{n>N_0}$ is tight on $L^s(0,T;L^2(D))$ for any $1\leq s<\infty$.
\end{lem}

\begin{proof}
Using Lemma \ref{lem_undelta-intBn_tight} and Lemma \ref{lem_intBn_tight} with the knowledge that $q'\leq 2$ we know that the sequence
\begin{align*}
\left(u_n=u_n-\int_0^\cdot B_n(s,u_n(s))\,dW_s+\int_0^\cdot B_n(s,u_n(s))\,dW_s\right)_{n>N_0}
\end{align*}
is bounded in $L^{q'}(\Omega;W^{\beta,2}(0,T;W^{-m,q'}(D))$ for any $\beta\in(0,\halbe)$. Furthermore, we obtain by Lemma \ref{proof_apriori_1} that $(u_n)_{n>N_0}$ is bounded in $L^p(\Omega;L^p(0,T;\sobolevp))\cap L^2(\Omega;L^2(0,T;L^2(D)))$.
By Lemma \ref{Lem_K_R_relativelycompact} $K_R$ is relatively compact in $L^{s}(0,T;\lz)$ for all R>0 and $1\leq s<\infty$. There holds for all $n\in\na,n>N_0,$ and an appropriate $R>0$
\begin{align*}
\mu_{u_n}(L^{s}(0,T;\lz)\setminus K_R)&=\int_{\left\{v\in L^{s}(0,T;\lz):\||v\||\geq R \right\}}1\,d\mu_{u_n}\\
&=\int_{\left\{\omega\in\Omega:\||u_n(\omega)\||\geq R \right\}}1\,d\mathbb{P}\\
&=\frac{1}{R^{q'}}\int_{\left\{\omega\in\Omega:\||u_n(\omega)\||\geq R \right\}} R^{q'}\,d\mathbb{P}\\
&\leq \frac{1}{R^{q'}}\int_\Omega\||u_n\||^{q'}\,d\mathbb{P}.
\end{align*}
\end{proof}

\subsection{Passage to the limit}\label{Sec_Stoch_Compactness}

For $n\in\na,n>N_0,$ we consider the vector 
\begin{align*}
Y_n:=\left(u_n,W,u_0\right) \quad\text{in}\quad \mathcal{X}=C([0,T];\lz)\times C([0,T];U)\times\lz.
\end{align*}
By Lemma~\ref{lem_undelta_tight} and Prokhorovs theorem  a not relabeld subsequence of $(u_n)_{n>N_0}$ converges in law for $n\rightarrow\infty$ to a probability measure $\mu_\infty$ with respect to $L^s(0,T;\lz)$ for all $1\leq s<\infty$.
Skorokhod's theorem implies the existence of
\begin{itemize}
\item a probability space $(\Omega',\mathcal{A}',\mathbb{P}')$ (which always can be chosen as $([0,1],\mathcal{B}([0,1]),\lambda)$ with $\mathcal{B}([0,1])$ the set of all Borel measures on $[0,1]$ and $\lambda$ the one-dimensional Lebesgue-measure)
\item a family of random variables $\mathcal{Y}_n=(v_n,\mathcal{W},v_0)$ on $(\Omega', \mathcal{A}',\mP')$ with values in $\mathcal{X}$ having the same law as $Y_n$
\item a random variable $u_\infty$ with values in $L^2(0,T;\lz)$ such that the law of $u_\infty$ is equal to the law of $u_n$ and for $n\rightarrow\infty$
\begin{align}\label{P-a.s.-conv_vndelta}
    v_n\rightarrow u_\infty \text{ in } L^s(0,T;\lz),\ \mP\text{-a.s. in }\Omega'.
\end{align}
\end{itemize}

\begin{remark}
    By \cite{BHR18} the random variables $\mathcal{W}$ and $v_0$ are independent of $n$.
\end{remark}

\begin{lem}\label{lem_conv_new_probabilityspace}
We have $v_n\in L^2(\Omega';C([0,T];\lz))$, $v_n(0)=v_0$ a.e. in $\Omega'\times D$, $\mathcal{W}(0)=0$ $\mP'$-a.s. in $\Omega'$, and the following convergences hold for $n\rightarrow\infty$ after passing to a not relabeled subsequence if necessary:
\begin{itemize}
    \item[$i)$] $v_n\rightarrow u_\infty$ in $L^\varrho(\Omega',L^s(0,T;\lz))$ for all $\varrho<2$ and all $1\leq s<\infty$
    \item[$ii)$] $\nabla v_n \rightharpoonup \nabla u_\infty$ in $L^p(\Omega';L^p(0,T;L^p(D)))$
    \item[$iii)$] $a(x,v_n,\nabla v_n)\rightharpoonup G$ in $L^{p'}(\Omega';L^{p'}(0,T;L^{p'}(D)))$ for an element $G\in L^{p'}(\Omega'; L^{p'}(0,T;\linebreak L^{p'}(D)))$
    \item[$iv)$] $f(v_n)\rightharpoonup f(u_\infty)$ in $L^2(\Omega';L^2(0,T;L^2(D)))$ and $f(v_n)\rightarrow f(u_\infty)$ in $L^\varrho(\Omega'; L^s(0,T;\linebreak L^2(D)))$.
\end{itemize}
\end{lem}

\begin{proof}
By equality in law we know $v_n\in L^2(\Omega';C([0,T];\lz))$, $v_n(0)=0$, and $\mathcal{W}(0)=0$.\\
$i)$ Since $(u_n)_{n>N_0}$ is bounded in $L^2(\Omega;L^2(0,T;L^2(D)))$ by Lemma~\ref{lem_sup_u_n_bound},
\begin{align}\label{vndelta_bounded}
    (v_n)_{n>N_0} \text{ is bounded in } L^2(\Omega';L^2(0,T;L^2(D)))
\end{align}
by equality in law. Using \eqref{P-a.s.-conv_vndelta} we obtain by Vitalis theorem the claimed convergence.\\
$ii)$ By Lemma~\ref{lemboundungradun} and equality in law
\begin{align}\label{nablavndelta_bounded}
    (\nabla v_n)_{n>N_0} \text{ is bounded in } L^p(\Omega';L^p(0,T;L^p(D))).
\end{align}
Hence, there exists a not relabeled subsequence such that $\nabla v_n \rightharpoonup \varphi$ in $L^p(\Omega';L^p(0,T;\linebreak L^p(D)))$ for $n\rightarrow\infty$ and an element $\varphi$ which can be verified as $\nabla u_\infty$.\\
$iii)$ Using \eqref{vndelta_bounded} and \eqref{nablavndelta_bounded} we can show by an analogous argumentation as in Lemma~\ref{a(u_n)_bounded} that $(a(\cdot,v_n,\nabla v_n))_{n>N_0}$ is bounded in $L^{p'}(\Omega';L^{p'}(0,T;L^{p'}(D)))$ and is hence weak convergent.\\
$iv)$ The convergences are a consequence of the Lipschitz continuity of $f$.
\end{proof}

\begin{defi}
For $t\in[0,T]$ and $n\in\na,n>N_0,$ we define $(F_t^{n})_{t\in[0,T]}$ to be the smallest sub-$\sigma$-field of $\mathcal{A}'$ generated by $v_n(s)$, $\mathcal{W}(s)$ for $0\leq s\leq t$, and $v_0$. The right-continuous, $\mathbb{P}'$-augmented filtration of $(F_t^n)_{t\in[0,T]}$ denoted by $(\mathcal{F}_t^n)_{t\in[0,T]}$ is for any $t\in[0,T]$ defined by
\begin{align*}
\mathcal{F}_t^n:=\bigcap_{s>t}\sigma[F_s^n\cup\{\mathcal{N}\in\mathcal{A}':\mathbb{P}'(\mathcal{N})=0\}].
\end{align*}
\end{defi}

\begin{remark}
From the previous definition it immediately follows that $v_0$ is $\mathcal{F}_0^n$-measurable for all $n>N_0$.
\end{remark}

\begin{lem}\label{lem_W_Q-Wiener_process}
For each $n\in\na,n>N_0,$ $v_n$ is adapted to $(\mathcal{F}_t^n)_{t\in[0,T]}$ and $\mathcal{W}=(\mathcal{W}(t))_{t\in[0,T]}$ is a $(\mathcal{F}_t^n)_{t\in[0,T]}$-adapted $Q$-Wiener process with values in $U$.
\end{lem}

\begin{proof}
Obviously, $v_n$ and $\mathcal{W}$ are $(\mathcal{F}_t^n)_{t\in[0,T]}$-adapted for any $n>N_0$ by construction of the filtration. By equality in law we know $\mathcal{W}(0)=0$ $\mP$-a.s. in $\Omega'$ and using Burkholder-Davis-Gundy inequality we have
\begin{align*}
    \erwws{\sup_{t\in[0,T]}|\mathcal{W}(t)|_U^2}=\erww{\sup_{t\in[0,T]}|W(t)|_U^2}\leq C_{BDG}\operatorname{Tr}(Q)T<\infty.
\end{align*}
Let $n\in\na,n>N_0,$ be arbitrary. For all $k\in\na,0\leq s\leq t\leq T$, and all bounded and continuous functions $\psi:\mathcal{X}\rightarrow\re$ we obtain by equality in law
\begin{align}\label{eqmW0}
    \erwws{\langle \mW(t)-\mW(s),e_k\rangle_U\psi\left((\mY_n)_{|[0,s]}\right)}=\erww{\langle W(t)-W(s),e_k\rangle_U\psi\left((Y_n)_{|[0,s]}\right)}=0,
\end{align}
where $(e_k)_{k\in\na}$ is an orthonormal basis of $U$. The real-valued random variable $\omega'\ni\Omega'\mapsto\psi\left((\mY_n)_{|[0,s]}(\omega')\right)$ is $(F_t^n)_{t\in[0,T]}$-measurable by definition. Using \eqref{eqmW0} we obtain for all $k\in\na,0\leq s\leq t\leq T$, and all bounded and continuous functions $\psi:\mathcal{X}\rightarrow\re$
\begin{align*}
    0&=\erwws{\langle \mW(t)-\mW(s),e_k\rangle_U\psi\left((\mY_n)_{|[0,s]}\right)}\\
    &=\erwws{\erwws{\langle \mW(t)-\mW(s),e_k\rangle_U\psi\left((\mY_n)_{|[0,s]}\right)}|F_s^n}\\
    &=\erwws{\psi\left((\mY_n)_{|[0,s]}\right)\erwws{\langle \mW(t)-\mW(s),e_k\rangle_U|F_s^n}}.
\end{align*}
The Doob-Dynkin lemma (see, e.g., \cite[Proposition 3]{RS06}) implies 
\begin{align*}
    \erwws{\mathds{1}_A\erwws{\langle\mW(t)-\mW(s),e_k\rangle_U|F_s^n}}=0
\end{align*}
for all $k\in\na,0\leq s\leq t\leq T$, and all $F_s^n$-measurable sets $A\in\mathcal{A'}$. Consequently, there holds
\begin{align*}
    \erwws{\langle\mW(t)-\mW(s),e_k\rangle_U|F_s^n}=0\quad\mP'\text{-a.s. in }\Omega'
\end{align*}
for all $0\leq s\leq t\leq T$ and $k\in\na$. Hence, $\mW$ is a $(F_t^n)_{t\in[0,T]}$-martingale, and, by \cite[p.75]{DM80}, $\mW$ is a martingale with respect to the augmented filtration $(\mathcal{F}_t^n)_{t\in[0,T]}$. Using equality in law of $\mW$ and $W$ we obtain for all $0\leq s\leq t\leq T$ and $k,j\in\na$
\begin{align*}
    0&=\erww{\langle W(t)-W(s),e_k\rangle_U\langle W(t)-W(s),e_j\rangle_U-\langle(t-s)Q(e_k),e_j\rangle_U\psi\left((Y_n)_{|[0,s]}\right)}\\
    &=\erwws{\langle\mW(t)-\mW(s),e_k\rangle_U\langle\mW(t)-\mW(s),e_j\rangle_U-\langle(t-s)Q(e_k),e_j\rangle_U\psi\left((\mY_n)_{|[0,s]}\right)}.
\end{align*}
With similar arguments as before, we get $\langle\langle W\rangle\rangle_t=tQ$ for all $t\in[0,T]$, see \cite[p.75]{DPZ14}, where $\langle\langle W\rangle\rangle$ denotes the quadratic variation process of $\mW$. By a generalized Levy's theorem (see \cite[Theorem 4.6]{DPZ14}) $\mW$ is a $Q$-Wiener process with values in $U$.
\end{proof}

\begin{lem}
For any $n\in\na,n>N_0$, and $t\in[0,T]$ we define 
\begin{align*}
    M_n(t):=v_n(t)-v_0+\int_0^t\left[\frac{1}{n} j(\vnd,\cdot)-\diver a(\cdot,v_n(s),\nabla v_n(s))+f(v_n(s))\right]ds.
\end{align*}
The stochastic process $(M_n(t))_{\tiT}$ is a square-integrable, continuous $(\mathcal{F}_t^n)_{\tiT}$-martingale with values in $L^2(D)$ such that for each $t\in[0,T]$
\begin{align}
    \langle\langle M_n\rangle\rangle_t&=\int_0^t\left(B_n(s,v_n(s)) Q^\halbe\right)\circ\left(B_n(s,v_n(s))Q^\halbe\right)^\ast\,ds\label{quadrVar_Mndelta}\\
    \langle\langle\mathcal{W},M_n\rangle\rangle_t&=\int_0^t Q\circ B_n(s,v_n(s))\,ds.\label{qudarVar_W_Mndelta}
\end{align}
\end{lem}

\begin{proof}
By Proposition~\ref{teo_existence_appr_solution} and equality in law we know that the stochastic process $(M_n(t))_{\tiT}$ has values in $L^2(D)$.
Moreover, we get by definition of $M_n$ and equality in law for all $n>N_0$
\begin{align}\label{law_Mndelta}
\begin{aligned}
    \mathcal{L}(M_n)&=\mathcal{L}\left(v_n-v_0+\int_0^\cdot\left[\frac{1}{n} j(v_n(s),\cdot)-\diver a(\cdot,v_n(s),\nabla v_n(s))+f(v_n(s))\right] ds\right)\\
    &=\mathcal{L}\left(u_n-v_0+\int_0^\cdot\left[\frac{1}{n} j(u_n(s),\cdot)-\diver a(\cdot,u_n(s),\nabla u_n(s))+f(u_n(s))\right] ds\right)\\
    &=\mathcal{L}\left(\int_0^\cdot B_n(s,u_n(s))\,dW_s\right),
\end{aligned}
\end{align}
where $\mathcal{L}(\cdot)$ denotes the law. Therefore, $(M_n(t))_{\tiT}$ is a martingale with respect to $(\mathcal{F}_t^n)_{\tiT}$ for all $n>N_0$, that can be shown with similar arguments as in the proof of Lemma \ref{lem_W_Q-Wiener_process}.\\
Let $n\in\na,n>N_0,$ be arbitrary. Because the mapping $(t,v)\mapsto B_n(t,v)$ is measurable on $(0,T)\times\lz$ by Corollary~\ref{lem_Lipschitz_B_n} $i)$ and $\mathcal{L}(u_n)=\mathcal{L}(v_n)$, we know $\mathcal{L}(B_n(\cdot,u_n))=\mathcal{L}(B_n(\cdot,v_n))$. Let $(e_k)_{k\in\na}$ be an orthonormal basis of $L^2(D)$. For all $k,j\in\na,0\leq s\leq t\leq T$, and all bounded and continuous functions $\psi:\mathcal{X}\rightarrow \re$ we get
\begin{align*}
    &\mathbb{E}'\bigg[\langle M_n(t)-M_n(s),e_k\rangle_{L^2}\langle M_n(t)-M_n(s),e_j\rangle_{L^2}\\
    &\quad-\langle\int_s^t\left(B_n(s,v_n(s))Q^\halbe\right)\left(B_n(s,v_n(s))Q^\halbe\right)^\ast(e_k)\,ds,e_j\rangle_{L^2}\,\psi\left((\mathcal{Y}_n)_{|[0,s]}\right)\bigg]\\
    &=\mathbb{E}\bigg[\langle\int_s^tB_n(r,u_n(r))\,dW_r,e_k\rangle_{L^2}\langle\int_s^tB_n(r,u_n(r))\,dW_r,e_j\rangle_{L^2}\\
    &\qquad-\langle\int_s^t\left(B_n(s,u_n(s))Q^\halbe\right)\left(B_n(s,u_n(s))Q^\halbe\right)^\ast(e_k)\,ds,e_j\rangle_{L^2}\,\psi\left((Y_n)_{|[0,s]}\right)\bigg]\\
    &=0.
\end{align*}
Consequently, we know for any $t\in[0,T]$
\begin{align*}
    \langle\langle M_n\rangle\rangle_t=\int_0^t\left(B_n(s,v_n(s))Q^\halbe\right)\circ\left(B_n(s,v_n(s))Q^\halbe\right)^\ast\,ds.
\end{align*}
Using \eqref{law_Mndelta} and $\mathcal{L}(\mW)=\mathcal{L}(W)$ we obtain for $\tiT$
\begin{align*}
    \langle\langle\mW,M_n\rangle\rangle_t=\int_0^t Q\circ B_n(s,v_n(s))\,ds.
\end{align*}
\end{proof}

\begin{lem}\label{lem_appr_Equ}
For all $n\in\na,n>N_0$, and all $t\in[0,T]$
\begin{align*}
    M_n(t)= \int_0^t B_n(s,v_n(s))\,d\mathcal{W}_s \text{ in }L^2(\Omega';L^2(D)).
\end{align*}
In particular, for all $t\in[0,T]$
\begin{align*}
    &v_n(t)+\int_0^t \frac{1}{n}\partial J(v_n(s))\,ds-\int_0^t\diver a(\cdot,v_n(s),\nabla v_n(s))\,ds+\int_0^t f(v_n(s))\,ds\\
    &=\int_0^t B_n(s,v_n(s))\,d\mW_s\quad\text{in }L^2(D), \text{ a.s. in }\Omega'.
\end{align*}
\end{lem}

\begin{proof}
Let $n\in\na,n>N_0,$ be arbitrary. Then, for all $t\in[0,T]$ we have 
\begin{align}\label{eq_identification_Mndelta_1}
    \erwws{\left\|M_n(t)-\int_0^t B_n(s,\vnd(s))\,d\mW_s\right\|_2^2}=\sum_{k\in\na}\erwws{\langle M_n(t)-\int_0^tB_n(s,\vnd(s))\,d\mW_s,e_k\rangle_{L^2}^2},
\end{align}
where for any $k\in\na$
\begin{align}\label{eq_identification_Mndelta_2}
\begin{aligned}
    &\erwws{\langle M_n(t)-\int_0^tB_n(s,\vnd(s))\,d\mW_s,e_k\rangle_{L^2}^2}\\
    &=\erwws{\langle M_n(t),e_k\rangle_{L^2}^2}-\erwws{\langle M_n(t),e_k\rangle_{L^2}\langle \int_0^t B_n(s,\vnd(s))\,d\mW_s,e_k\rangle_{L^2}}\\
    &\quad+\erwws{\langle \int_0^tB_n(s,\vnd(s))\,d\mW_s,e_k\rangle_{L^2}^2}.
\end{aligned}
\end{align}
By \eqref{quadrVar_Mndelta} and \eqref{law_Mndelta} we know for all $\tiT$
\begin{align*}
    &\sum_{k\in\na}\erwws{\langle\int_0^t B_n(s,\vnd(s))\,d\mW_s,e_k\rangle_{L^2}^2}\\
    &=\sum_{k\in\na}\erwws{\langle M_n(t),e_k\rangle_{L^2}^2}\\
    &=\erwws{\int_0^t\operatorname{Tr}\left[\left(B_n(s,\vnd(s))Q^\halbe\right)\circ\left(B_n(s,\vnd(s))Q^\halbe\right)^\ast \right]\,ds }.
\end{align*}
Using \eqref{qudarVar_W_Mndelta} we further obtain for $t\in[0,T]$
\begin{align*}
    &\sum_{k\in\na}\erwws{\langle M_n(t),e_k\rangle_{L^2}\langle\int_0^t B_n(s,\vnd(s))\,d\mW_s,e_k\rangle_{L^2}}\\
    &=\erwws{\operatorname{Tr}\left[\langle\langle M_n(\cdot),\int_0^\cdot B_n(s,\vnd(s))\,d\mW_s\rangle\rangle_t\right]}\\
    &=\erwws{\int_0^t\operatorname{Tr}\left[\left(B_n(s,\vnd(s))Q^\halbe\right)\circ\left(B_n(s,\vnd(s))Q^\halbe\right)^\ast\right]\,ds }.
\end{align*}
Therefore we get from \eqref{eq_identification_Mndelta_1} and \eqref{eq_identification_Mndelta_2} for all $\tiT$
\begin{align*}
    \erwws{\left\|M_n(t)-\int_0^tB_n(s,\vnd(s))\,d\mW_s\right\|_2^2}=0.
\end{align*}
\end{proof}

\begin{lem}
The filtration $(\mathcal{F}_t^n)_{t\in[0,T]}$ can be chosen independently of $n$.
\end{lem}

\begin{proof}
Let $(\mF_t)_{t\in[0,T|}$ be the smallest filtration in $\mathcal{A}'$ generated by $v_0$ and $\mW(s)$ for $0\leq s\leq t\leq T$ and augmented in order to satisfy the usual assumptions. One may show as before that $\mW$ is a $Q$-Wiener process with respect to $\mFt$. Applying the arguments of Section~\ref{Sec_ApprEqu} to the stochastic basis $(\Omega',\mathcal{A}',\mP,\mFt)$, associated with $\mW$, there exists for any $n\in\na$ a unique solution $\widetilde{u}_n$ to the approximated equation with the $\widetilde{\mathcal{F}}_0$-measurable initial datum $v_0$. By uniqueness $\widetilde{u}_n=v_n$.
\end{proof}

\begin{lem}\label{lem_conv_int_Bn}
    For all $t\in[0,T]$ the convergence
    \begin{align*}
        \int_0^t B_n(s,\vnd(s))\,d\mW_s\rightarrow\int_0^t B(s,u_\infty(s))\,d\mW_s\text{ in }L^2(\Omega';L^2(D))
    \end{align*}
    holds true for $n\rightarrow\infty$.
\end{lem}

\begin{proof}
For all $n\in\na,n>N_0$, a.e. $t\in(0,T)$, $\mP'$-a.s. in $\Omega'$ we know by applying Parseval identitiy
\begin{align*}
    &\|B_n(t,v_n(t))-B(t,\vnd(t))\|_{\HS}^2\\
    &=\sum_{k\in\na}\int_D\left|\int_D\left[\sigma_n(t,\vnd(t,x))-\sigma(t,\vnd(t,x))\right]k(x,y)e_k(y)\,dy\right|^2 dx\\
    &=\int_D|\sigma_n(t,\vnd(t,x))-\sigma(t,\vnd(t,x))|^2\sum_{k\in\na}|\langle k(x,\cdot),e_k\rangle_{L^2}|^2\,dx\\
    &=\int_D|\sigma_n(t,\vnd(t,x))-\sigma(t,\vnd(t,x))|^2\|k(x,\cdot)\|_2^2\,dx\\
    &\leq C_k\|\sigma_n(t,\vnd(t))-\sigma(t,\vnd(t))\|_2^2.
\end{align*}
Note that for all $n\in\na,n>N_0,$ and a.e. $(t,x)\in(0,T)\times D$, $\mP'$-a.s. in $\Omega'$ by using (S2a) and Proposition \ref{231124_p1} $ii)$
\begin{align*}
    |\sigma_n(t,\vnd(t,x))-\sigma(t,\vnd(t,x))|
    &=\sup_{\mu\in\re}\left\{\sigma(t,\vnd(t,x)-\sigma(t,\mu)-n|\vnd(t,x)-\mu|\right\}\\
    &\leq \sup_{\mu\in\re}\left\{L_\alpha |\vnd(t,x)-\mu|^\alpha-n|\vnd(t,x)-\mu|\right\}\\
    &\leq\sup_{r\in[0,\infty)}\left\{L_\alpha r^\alpha-nr\right\},
\end{align*}
where this last term converges to zero as shown in the proof of Proposition \ref{231124_p1} $iv)$. Moreover, we know for all $n>N_0$, a.e. $(t,x)\in(0,T)\times D$, $\mP'$-a.s. in $\Omega'$
\begin{align*}
    |\sigma_n(t,\vnd(t,x))-\sigma(t,\vnd(t,x))|^2\leq\max_{r\in[0,\infty)}\{L_\alpha r^\alpha-nr\}^2\leq\frac{(1-\alpha)^2}{L^\alpha \alpha^{\frac{1+\alpha}{\alpha-1}}},
\end{align*}
see proof of Proposition \ref{231124_p1} $iv)$. From Lebesgues dominated convergence theorem we obtain for $n\rightarrow\infty$
\begin{align}\label{eq_conv_B_n_vnd_1}
    \erwws{\int_0^T\int_D|\sigma_n(t,\vnd(t,x))-\sigma(t,\vnd(t,x))|^2\,dx\,dt}\rightarrow0.
\end{align}
Since for any $\alpha\in(0,1)$ there exist $\varrho<2$ and $1\leq s<\infty$ such that $L^\varrho(\Omega';L^s(0,T;L^2(D)))\linebreak\hookrightarrow L^{2\alpha}(\Omega';L^{2\alpha}(0,T;L^{2\alpha}(D)))$ holds true we have
\begin{align}\label{eq_conv_B_n_vnd_2}
    \erwws{\int_0^T\|B(t,\vnd(t))-B(t,u_\infty(t))\|_{\HS}^2\,dt}\leq C_kL_\alpha^2\erwws{\int_0^T\|\vnd(t)-u_\infty(t)\|_{2\alpha}^{2\alpha}\,dt}\rightarrow 0
\end{align}
for $n\rightarrow\infty$ by Lemma \ref{lem_conv_new_probabilityspace} $i)$. Therefore, we obtain by \eqref{eq_conv_B_n_vnd_1} and \eqref{eq_conv_B_n_vnd_2}
\begin{align}\label{conv_Bn(vnd)_to_B(uinfty)}
    B_n(\cdot,\vnd)\rightarrow B(\cdot,u_\infty)\text{ in } L^2(\Omega';L^2(0,T;\HS(L^2(D))).
\end{align}
Using Burkholder-Davis-Gundy inequality we get for $n\rightarrow\infty$
\begin{align*}
    \int_0^\cdot B_n(s,\vnd(s))\,d\mW_s\rightarrow \int_0^\cdot B(s,u_\infty(s))\,d\mW_s\text{ in }L^2(\Omega';C([0,T];\lz))
\end{align*}
and for all $\tiT$
\begin{align*}
    \int_0^t B_n(s,\vnd(s))\,d\mW_s\rightarrow\int_0^t B(s,u_\infty(s))\,d\mW_s\text{ in }L^2(\Omega';\lz).
\end{align*}
\end{proof}

\begin{prop}\label{lem_limit_Equ_without_identification}
The function $u_\infty$ is a $\mFt$-adapted, square integrable stochastic process with continuous paths in $\lz$ such that $u_\infty(0)=v_0$. Moreover, $u_\infty\in L^p(\Omega';L^p(0,T;W^{1,p}_0(D)))$ and
\begin{align*}
    \partial_t\left(u_\infty-\int_0^\cdot B(s,u_\infty(s))\,d\mW_s\right)-\diver G+f(u_\infty)=0 \text{ in }L^{p'}(\Omega;L^{p'}(0,T;W^{-1,p'}(D))).
\end{align*}
\end{prop}

\begin{remark}
    If $u_\infty$ is given as in Proposition~\ref{lem_limit_Equ_without_identification}, in particular we have for all $\tiT$
    \begin{align*}
        u_\infty(t)-v_0-\int_0^t G(s)-f(u_\infty(s))\,ds=\int_0^tB(s,u_\infty(s))\,d\mW_s \text{ in }L^2(D)\text{ a.s. in }\Omega'.
    \end{align*}
\end{remark}

\begin{proof}
By \eqref{u_n-B_n_bounded} and equality in law we know that there exists a not relabeled subsequence such that for $n\rightarrow\infty$
\begin{align}\label{conv_vnd-int_B_n_in_C}
    \vnd-\int_0^\cdot B_n(s,\vnd(s))\,d\mW_s\rightharpoonup u_\infty-\int_0^\cdot B(s,u_\infty(s))\,d\mW_s\text{ in }L^{q'}(\Omega';\mathfrak{V}).
\end{align}
Since $\mathfrak{V}$ is continuously and densely embedded into $C([0,T];\sobolevmqs)$, the weak convergence holds true in $L^{q'}(\Omega';C([0,T];\sobolevmqs))$. In particular, by Lemma \ref{lem_conv_int_Bn}, $\vnd\rightharpoonup u_\infty$ in $L^{q'}(\Omega';C([0,T];\sobolevmqs))$ and hence $\vnd(t)\rightharpoonup u_\infty(t)$ in $L^{q'}(\Omega';\sobolevmqs)$ for all $\tiT$. Because $(\vnd)_{n>N_0}$ is bounded in $L^2(\Omega';C([0,T];\lz))$ by Lemma \ref{lem_sup_u_n_bound} and equality in law, it follows that $(\vnd(t))_{n>N_0}$ is bounded in $L^2(\Omega';L^2(D))$ and we may conclude $\vnd(t)\rightharpoonup u_\infty(t)$ in $L^2(\Omega';L^2(D))$ for all $\tiT$ up to a subsequence which may depend on $\tiT$. Therefore, we have $u_\infty(0)=v_0$.\\
Let $A\in\mathcal{A}',\xi\in C_c^\infty((0,T))$, and $\varphi\in C_c^\infty(D)$, then for all $n\in\na, n>N_0,$ there holds by Lemma \ref{lem_appr_Equ}
\begin{align*}
    0=&\int_A\int_0^T\xi(t)\langle\partial_t\left(\vnd(t)-\int_0^t B_n(s,\vnd(s))\,d\mW_s\right),\varphi\rangle_{q',q}\,dt\,d\mP'\\
    &+\int_A\int_0^T\xi(t)\frac{1}{n} j(\vnd(t),\varphi)\,dt\,d\mP'+\int_A\int_0^T\int_D\xi(t) a(x,\vnd(t),\nabla\vnd(t))\cdot\nabla\varphi\,dx\,dt\,d\mP'\\
    &+\int_A\int_0^T\int_D \xi(t)f(\vnd(t))\varphi\,dt\,d\mP'\\
    =: &I_1+I_2+I_3+I_4.
\end{align*}
Using partial integration (see \cite[Proposition 2.5.2]{D01}), Lemma \ref{lem_conv_new_probabilityspace} $i)$, and \eqref{conv_vnd-int_B_n_in_C} we get
\begin{align}\label{eq_conv_I1}
\begin{aligned}
    I_1&=\int_A\langle\int_0^T\xi(t)\partial_t\left(\vnd(t)-\int_0^t B_n(s,\vnd(s))\,d\mW_s\right)\,dt,\varphi\rangle_{q',q}\,d\mP'\\
    &=-\int_A\int_D\int_0^T\xi'(t)\left(\vnd(t)-\int_0^tB_n(s,\vnd(s))\,d\mW_s\right)\varphi\,dt\,dx\,d\mP'\\
    &\rightarrow -\int_A\int_D\int_0^T\xi'(t)\left(u_\infty(t)-\int_0^t B(s,u_\infty(s))\,d\mW_s\right)\varphi\,dt\,dx\,d\mP'\\
    &=\int_A\int_0^T\xi(t)\langle\partial_t\left(u_\infty(t)-\int_0^tB(s,u_\infty(s))\,d\mW_s\right),\varphi\rangle_{q',q}\,dt\,d\mP'.
\end{aligned}
\end{align}
Moreover, there holds for all $n>N_0$
\begin{align*}
    I_2&=\int_A\int_0^T\xi(t)\frac{1}{n} \left[(\vnd(t),\varphi)_{H_0^m}+\int_D\sum_{|\gamma|\leq m}|D^\gamma \vnd(t)|^{q-2}D^\gamma\vnd(t)\cdot D^\gamma\varphi\,dx\right]\,dt\,d\mP'\\
    &\leq \|\xi\|_\infty\int_A\int_0^T\frac{1}{n} \left[\|\vnd(t)\|_{H_0^m}\|\varphi\|_{H_0^m}+\|\vnd(t)\|_{W_0^{m,q}}^{q-1}\|\varphi\|_{W^{m,q}_0}\right]\,dt\,d\mP'\\
    &\leq \|\xi\|_\infty\|\varphi\|_{W^{m,q}_0}\int_A\int_0^TC_E \frac{1}{n^\halbe} \frac{1}{n^\halbe} \|\vnd(t)\|_{H_0^m}+\frac{1}{n^{\frac{1}{q}}} \frac{1}{n^\frac{1}{q'}} \|\vnd(t)\|_{W^{m,q}_0}^{q-1}\,dt\,d\mP'\\
    &\leq\|\xi\|_\infty\|\varphi\|_{W^{m,q}_0}\Bigg[C_E\frac{1}{n^\halbe} \left(\erwws{\int_0^T\frac{1}{n} \|\vnd(t)\|_{H_0^m}^2\,dt}\right)^\halbe\\
    &\quad+\frac{1}{n^{\frac{1}{q}}} \left(\erwws{\int_0^T\frac{1}{n} \|\vnd(t)\|_{W^{m,q}_0}^q\,dt}\right)^{\frac{1}{q'}} \Bigg].
\end{align*}
By Lemma \ref{lemboundungradun} and equality in law $\erwws{\int_0^T\frac{1}{n} \|\vnd(t)\|_{H_0^m}^2\,dt}$ and $\erwws{\int_0^T\frac{1}{n} \|\vnd(t)\|_{W^{m,q}_0}^q\,dt}$ are bounded by a constant independent of $n$. Hence, we obtain
\begin{align}\label{conv_delta_j}
\lim_{n\rightarrow\infty}I_2=0.
\end{align}
Lemma \ref{lem_conv_new_probabilityspace} $iii),iv)$, \eqref{eq_conv_I1}, and \eqref{conv_delta_j} provide for all $A\in\mathcal{A}',\xi\in C_c^\infty((0,T)),$ and $\varphi\in C_c^\infty(D)$
\begin{align}\label{eq_tested_limit_equation}
\begin{aligned}
    0=&\int_A\int_0^T\xi(t)\langle\partial_t\left(u_\infty(t)-\int_0^t B(s,u_\infty(s))\,d\mW_s\right),\varphi\rangle_{q',q}\,dt\,d\mP'\\
    &+\int_A\int_0^T\int_D\xi(t) G(t)\cdot\nabla\varphi\,dx\,dt\,d\mP'+\int_A\int_0^T\int_D \xi(t)f(u_\infty(t))\varphi\,dt\,d\mP'.
\end{aligned}
\end{align}
We already know that $\nabla \vnd\rightharpoonup \nabla u_\infty$ in $L^p(\Omega';L^p(0,T;L^p(D)))$ and in particular $u_\infty\in L^p(\Omega';L^p(0,T;W^{1,p}_0(D)))$. Now, from equation \eqref{eq_tested_limit_equation} it follows that $u_\infty\in L^{\min\{p',2\}}(\Omega';\linebreak C([0,T];W^{-1,p'}(D)))$ and that
\begin{align}\label{eq_limit_eq_on_Omegaprime}
    u_\infty(t)-u_0-\int_0^t\diver G(s)-f(u_\infty(s))\,ds=\int_0^t B(s,u_\infty(s))\,d\mW_s\text{ in }W^{-1,p'}(D)
\end{align}
a.s. in $\Omega'$ for all $\tiT$. From \eqref{eq_limit_eq_on_Omegaprime} by \cite[Theorem 4.2.5]{LR15} $u_\infty$ is a $\mFt$-adapted, square-integrable stochastic process with continuous paths in $\lz$ and Itô's formula holds.
\end{proof}

\begin{lem}\label{lem_identification_G}
There holds
\begin{align*}
    G=a(\cdot,u_\infty,\nabla u_\infty)\text{ in }L^{p'}(\Omega';L^{p'}(0,T;L^{p'}(D))).
\end{align*}
\end{lem}

\begin{proof}
Let $\beta\in\re$ and $t\in[0,T]$ be arbitrary. Applying Itô's formula and taking the expectation we obtain from Lemma \ref{lem_appr_Equ}
\begin{align}\label{eq_Ito_vnd}
\begin{aligned}
    &\frac{1}{2}\erwws{\|\vnd(t)\|_2^2}-\halbe\erwws{\|v_0\|_2^2}+\erwws{\int_0^t\int_D a(x,\vnd(s),\nabla\vnd(s))\,dx\,ds}\\
    &+\frac{1}{n} \erwws{\int_0^t j(\vnd(s),\vnd(s))\,ds}+\erwws{\int_0^t\int_D f(\vnd(s))\vnd(s)\,dx\,ds}\\
    =\,&\halbe\erwws{\int_0^t\|B_n(s,\vnd(s))\|_{\HS}^2\,ds}.
\end{aligned}
\end{align}
On the other hand, we obtain from Proposition~\ref{lem_limit_Equ_without_identification} by applying Itô's formula and taking the expectation
\begin{align}\label{eq_Ito_uinfty}
\begin{aligned}
    &\halbe\erwws{\|u_\infty(t)\|_2^2}-\halbe\erwws{\|v_0\|_2^2}+\erwws{\int_0^t\int_D G(x,s)\cdot\nabla u_\infty(s)\,dx\,ds}\\
    &+\erwws{\int_0^t\int_D f(u_\infty(s))u_\infty(s)\,dx\,ds}=\halbe\erwws{\int_0^t\|B(s,u_\infty(s))\|_{\HS}^2\,ds}.
\end{aligned}
\end{align}
Taking the difference \eqref{eq_Ito_vnd}-\eqref{eq_Ito_uinfty} we get
\begin{align*}
    &\halbe\erwws{\|\vnd(t)\|_2^2}-\halbe\erwws{\|u_\infty(t)\|_2^2}\\
    &+\erwws{\int_0^t\int_D a(x,\vnd(s),\nabla\vnd(s))\cdot\nabla \vnd(s)-G(x,s)\cdot\nabla u_\infty(s)\,dx\,ds}\\
    &+\erwws{\int_0^t\int_D f(\vnd(s))\vnd(s)-f(u_\infty(s))u_\infty(s)\,dx\,ds}\\
    &\leq\erwws{\int_0^t\|B_n(s,\vnd(s))\|_{\HS}^2-\|B(s,u_\infty(s))\|_{\HS}^2\,ds}.
\end{align*}
Using Lemma \ref{lem_conv_new_probabilityspace} $i)$, the fact that $f\in L^\infty(\re)$, and \eqref{conv_Bn(vnd)_to_B(uinfty)}, we obtain
\begin{align*}
    \limsup_{n\rightarrow\infty}\bigg(&\erwws{\|\vnd(t)\|_2^2}-\erwws{\|u_\infty(t)\|_2^2}\\
    &+\erwws{\int_0^t\int_D a(x,\vnd(s),\nabla\vnd(s))\cdot\nabla\vnd(s)-G(x,s)\cdot\nabla u_\infty(s)\,dx\,ds} \bigg)\\
    \leq 0.
\end{align*}
Therefore, using the lower semicontinuity of the norm, there holds
\begin{align*}
    &\limsup_{n\rightarrow\infty}\erwws{\int_0^t\int_D a(x,\vnd(s),\nabla \vnd(s))\cdot\nabla\vnd(s)\,dx\,ds}\\
    &\leq \erwws{\int_0^t\int_D G(x,s)\cdot\nabla u_\infty(s)\,dx\,ds}.
\end{align*}
Applying a stochastic version of Minty's trick (see \cite[Lemma 8.8]{R13}) provides\linebreak $G=a(\cdot,u_\infty,\nabla u_\infty)$.
\end{proof}

\section{Pathwise uniqueness of solutions}\label{Sec_Uniqueness}
In this section, we prove Theorem \ref{maintheorem_uniqueness}.

\begin{proof}
Let $\eps>0$ and $\eta_\eps$ be a nondecreasing, Lipschitz continuous approximation of the sign-function defined by $\eta_\eps(r):=2\int_0^r\frac{1}{\eps}\rho(\frac{s}{\eps})\,ds$ for $r\in\re$, where $\rho(s):=c\operatorname{exp}\left(\frac{1}{s^2-1}\right)\mathds{1}_{\{|s|\leq 1\}}$ such that $\int_\re\rho(s)\,ds=1$ is a classical mollifier approximation of the Dirac measure with support on $[-\eps,\eps]$ (see \cite[p.195]{VZ19}). Define for $r\in\re, u\in L^2(D)$
\begin{align*}
N_\eps(r):=\int_0^r\eta_\eps(s)\,ds\quad\text{and}\quad F_\eps(u):=\int_D N_\eps(u(x))\,dx.
\end{align*}
Note that one can show by easy calculation $|N_\eps''(r)|\leq\frac{2c}{\eps}\mathds{1}_{\{|r|\leq\eps\}}$.
Because $u_1$ and $u_2$ are both solutions to \eqref{spde} we have for any $\tiT$
\begin{align*}
&u_1(t)-u_2(t)-(u_0^1-u_0^2)-\int_0^t\diver(a(\cdot,u_1(s),\nabla u_1(s))-a(\cdot,u_2(s),\nabla u_2(s)))\,ds\\
&+\int_0^t f(u_1(s))-f(u_2(s))\,ds = \int_0^t B(s,u_1(s))-B(s,u_2(s))\,dW_s.
\end{align*}
Applying Itô's fomula to this stochastic process by using $F_\eps$ (see \cite[p.78]{P21}) provides for any $t\in[0,T]$ after taking the expectation
\begin{align}\label{eq_uniq_Fdelta}
\begin{aligned}
&\erww{F_\eps(u_1(t)-u_2(t))}-\erww{F_\eps(u_0^1-u_0^2)}\\
&-\erww{\int_0^t\langle\diver\left(a(\cdot,u_1(s),\nabla u_1(s))-a(\cdot,u_2(s),\nabla u_2(s))\right),N_\eps'(u_1(s)-u_2(s))\rangle_{p',p}\,ds}\\
&+\erww{\int_0^t\int_D(f(u_1(s))-f(u_2(s)))N_\eps'(u_1(s)-u_2(s))\,dx\,ds}\\
&=\halbe\mathbb{E}\bigg[\int_0^t\operatorname{Tr}\Big[F_\eps''(u_1-u_2)(s))(B(s,u_1(s))-B(s,u_2(s)))Q\\
&\qquad(B(s,u_1(s))-B(s,u_2(s)))^*\Big]\,ds\bigg]\\
&\Leftrightarrow I_1+I_2+I_3+I_4=I_5.
\end{aligned}
\end{align}
By using (A1) we know that for any $\tiT$
\begin{align*}
I_2
&=\erww{\int_0^t\int_D (a(x,u_1,\nabla u_1)-a(x,u_2,\nabla u_2))\nabla(u_1-u_2)N_\eps''(u_1-u_2)\,dx\,ds}\\
&\geq \erww{\int_0^t\int_D(a(x,u_1,\nabla u_2)-a(x,u_2,\nabla u_2))\nabla (u_1-u_2)N_\eps''(u_1-u_2)\,dx\,ds}.
\end{align*}
By using (A3) and Hölder inequality there holds for any $t\in[0,T]$
\begin{align*}
&\left|\erww{\int_0^t\int_D(a(x,u_1,\nabla u_2)-a(x,u_2,\nabla u_2))\nabla (u_1-u_2)N_\eps''(u_1-u_2)\,dx\,ds} \right|\\
&\leq \erww{\int_0^t\int_D(C_5|\nabla u_2|^{p-1}+h(x))|u_1-u_2||\nabla(u_1-u_2)|N_\eps''(u_1-u_2)\,dx\,ds}\\
&\leq 2c\erww{\int_0^t\int_D(C_5|\nabla u_2|^{p-1}+h(x))|\nabla(u_1-u_2)|\mathds{1}_{\{|u_1-u_2|\leq\eps\}}\,dx\,ds}\\
&\leq 2cC_5\left(\erww{\int_0^t\|\nabla u_2\|_p^p\,ds}\right)^{p-1}\left(\erww{\int_0^t\int_D\mathds{1}_{\{|u_1-u_2|\leq\eps\}}|\nabla(u_1-u_2)|^p\,dx\,ds}\right)^{\frac{1}{p}}\\
&\quad+ 2c\|h\|_{p'}\left(\erww{\int_0^t\int_D\mathds{1}_{\{|u_1-u_2|\leq\eps\}}|\nabla(u_1-u_2)|^p\,dx\,ds}\right)^{\frac{1}{p}}\\
&\rightarrow 0\quad\text{for }\eps\downarrow0.
\end{align*}
Hence,
\begin{align}\label{uniq_I2_geq0}
\liminf_{\eps\downarrow0} I_2\geq0.
\end{align}
Since $\eta_\eps$ is an approximation of the sign function, we obtain by the Lipschitz continuity of $f$ for all $\tiT$
\begin{align}\label{uniq_f}
\begin{aligned}
    \lim_{\eps\downarrow0}I_3&=\lim_{\eps\downarrow0}\erww{\int_0^t\int_D(f(u_1)-f(u_2))\eta_\eps(u_1-u_2)\,dx\,ds}\\
    &=\erww{\int_0^t\int_D(f(u_1)-f(u_2))\operatorname{sign}(u_1-u_2)\,dx\,ds}\\
    &\leq\erww{\int_0^t\int_D L_f|u_1-u_2|\,dx\,ds}.
\end{aligned}
\end{align}
Let $(e_k)_{k\in\na}$ be an orthonormal basis of $U$ consisting of eigenvectors of $Q$. By using \cite[Proposition B.0.10]{LR15} there holds for all $\tiT$
\begin{align*}
    I_5&=\halbe\erww{\int_0^t\operatorname{Tr}\left[\left((B(s,u_1)-B(s,u_2))Q^\halbe\right)^*F_\eps''(u_1-u_2)(B(s,u_1)-B(s,u_2))Q^\halbe \right]\,ds}\\
    &=\halbe \mathbb{E}\bigg[\int_0^t\sum_{k\in\na}\langle \left(B(s,u_1)-B(s,u_2))Q^\halbe\right)^*F_\eps''(u_1-u_2)\\
    &\qquad(B(s,u_1)-B(s,u_2))Q^\halbe(e_k),e_k \rangle_U\,ds\bigg]\\
    &=\halbe\mathbb{E}\bigg[\int_0^t\sum_{k\in\na}\langle F_\eps''(u_1-u_2)(B(s,u_1)-B(s,u_2))Q^\halbe(e_k),\\
    &\qquad(B(s,u_1)-B(s,u_2))Q^\halbe(e_k)\rangle_{L^2}\,ds\bigg]\\
    &=\halbe\erww{\int_0^t\int_D N_\eps''(u_1-u_2)\sum_{k\in\na}\left|(B(s,u_1)-B(s,u_2))Q^\halbe(e_k)\right|^2\,dx\,ds}.
\end{align*}
Consequently, for all $\tiT$ there holds
\begin{align*}
    |I_5|&\leq\halbe \erww{\int_0^t\int_DN_\eps''(u_1-u_2)\sum_{k\in\na}\left|\int_D(\sigma(s,u_1(x))-\sigma(s,u_2(x))k(x,y)f_k(y)\,dy\right|^2 dx\,ds}\\
    &\leq \halbe\erww{\int_0^t\int_D N_\eps''(u_1-u_2)L_\sigma^2|u_1-u_2|^{2\alpha}\sum_{k\in\na}\left(\int_D|k(x,y)||Q^\halbe(e_k)(y)|\,dy \right)^2 dx\,ds}\\
    &\leq \frac{L_\sigma^2}{2}\erww{\int_0^t\int_DN_\eps''(u_1-u_2)|u_1-u_2|^{2\alpha}\|k(x,\cdot)\|_2^2\left(\sum_{k\in\na}\|Q^\halbe(e_k)\|_2^2\right)\,dx\,ds}.
\end{align*}
Since $(e_k)_{k\in\na}$ are eigenvectors of $Q$, there exist $(\lambda_k)_{k\in\na}$ such that
\begin{align*}
    \sum_{k\in\na}\|Q^\halbe(e_k)\|_2^2=\sum_{k\in\na}\|\lambda_k e_k\|_2^2=\sum_{k\in\na}|\lambda_k|\leq C
\end{align*}
for a constant $C\geq0$. Therefore, we obtain for any $\tiT$
\begin{align}\label{uniq_lim_I5}
\begin{aligned}
    |I_5|&\leq \frac{CL_\sigma^2}{2}\erww{\int_0^t\int_D N_\eps''(u_1-u_2)|u_1-u_2|^{2\alpha}\|k(x,\cdot)\|_2^2\,dx\,ds}\\
    &\leq \frac{CL_\sigma^2 c}{\eps}\erww{\int_0^t\int_D\mathds{1}_{\{|u_1-u_2|\leq\eps\}}|u_1-u_2|^{2\alpha}\|k(x,\cdot)\|_2^2\,dx\,ds}
\end{aligned}
\end{align}
If $\alpha\in(\halbe,1)$, we can estimate by \eqref{uniq_lim_I5}
\begin{align*}
    |I_5|\leq CL_\sigma^2cT\|k\|_{L^2(D\times D)}\eps^{2\alpha-1}\rightarrow 0\quad\text{for }\eps\downarrow 0.
\end{align*}
For $\alpha=\halbe$ we find by \eqref{uniq_lim_I5}
\begin{align*}
    |I_5|&\leq \frac{CL_\sigma^2 c}{\eps}\erww{\int_0^t\int_D\mathds{1}_{\{|u_1-u_2|\leq\eps\}}\mathds{1}_{\{u_1\neq u_2\}}|u_1-u_2|\|k(x,\cdot)\|_2^2\,dx\,ds}\\
    &\leq CL_\sigma^2 c\erww{\int_0^t\int_D\mathds{1}_{\{|u_1-u_2|\leq\eps\}}\mathds{1}_{\{u_1\neq u_2\}}\|k(x,\cdot)\|_2^2\,dx\,ds}\\
    &\rightarrow 0\text{ for }\eps\downarrow0.
\end{align*}
Since there holds
\begin{align*}
    &\lim_{\eps\downarrow0}\left(\erww{F_\eps(u_1(t)-u_2(t))}-\erww{F_\eps(u_0^1-u_0^2)}\right)\\
    &=\erww{\int_D|u_1(t)-u_2(t)|\,dx}-\erww{\int_D|u_0^1-u_0^2|\,dx}
\end{align*}
we obtain from \eqref{eq_uniq_Fdelta} by using \eqref{uniq_I2_geq0}, \eqref{uniq_f}, and \eqref{uniq_lim_I5}
\begin{align*}
    \erww{\int_D|u_1(t)-u_2(t)|\,dx}\leq \erww{\int_D|u_0^1-u_0^2|\,dx}+L_f\erww{\int_0^t\int_D|u_1(s)-u_2(s)|\,dx\,ds}.
\end{align*}
for all $\tiT$. Using Gronwalls lemma we get
\begin{align*}
    \erww{\int_D|u_1(t)-u_2(t)|\,dx}\leq e^{Lt}\erww{\int_D|u_0^1-u_0^2|\,dx}.
\end{align*}
\end{proof}

\textbf{Acknowledgment:} This work has been supported by the German Research Foundation project ZI 1542/3-1.

\bibliographystyle{abbrv}
\bibliography{Hoelder_Bibliography}

\begin{thebibliography}{10}

\bibitem{ABM06}
H.~Attouch, G.~Buttazzo, and G.~Michaille.
\newblock {\em Variational analysis in {S}obolev and {BV} spaces}, volume~6 of
  {\em MPS/SIAM Series on Optimization}.
\newblock Society for Industrial and Applied Mathematics (SIAM), Philadelphia,
  PA; Mathematical Programming Society (MPS), Philadelphia, PA, 2006.
\newblock Applications to PDEs and optimization.

\bibitem{BDR16}
V.~Barbu, G.~Da~Prato, and M.~R\"{o}ckner.
\newblock {\em Stochastic porous media equations}, volume 2163 of {\em Lecture
  Notes in Mathematics}.
\newblock Springer, [Cham], 2016.

\bibitem{BNSZ23}
C.~Bauzet, F.~Nabet, K.~Schmitz, and A.~Zimmermann.
\newblock Convergence of a finite-volume scheme for a heat equation with a
  multiplicative {L}ipschitz noise.
\newblock {\em ESAIM Math. Model. Numer. Anal.}, 57(2):745--783, 2023.

\bibitem{BBNP14}
L.~Ba\v{n}as, Z.~Brze\'{z}niak, M.~Neklyudov, and A.~Prohl.
\newblock A convergent finite-element-based discretization of the stochastic
  {L}andau-{L}ifshitz-{G}ilbert equation.
\newblock {\em IMA J. Numer. Anal.}, 34(2):502--549, 2014.

\bibitem{BBNP14book}
L.~Ba\v{n}as, Z.~Brze\'{z}niak, M.~Neklyudov, and A.~Prohl.
\newblock {\em Stochastic ferromagnetism}, volume~58 of {\em De Gruyter Studies
  in Mathematics}.
\newblock De Gruyter, Berlin, 2014.

\bibitem{BV19}
F.~Berthelin and J.~Vovelle.
\newblock Stochastic isentropic {E}uler equations.
\newblock {\em Ann. Sci. \'{E}c. Norm. Sup\'{e}r. (4)}, 52(1):181--254, 2019.

\bibitem{BFH18}
D.~Breit, E.~Feireisl, and M.~Hofmanov\'{a}.
\newblock {\em Stochastically forced compressible fluid flows}, volume~3 of
  {\em De Gruyter Series in Applied and Numerical Mathematics}.
\newblock De Gruyter, Berlin, 2018.

\bibitem{BHR18}
Z.~Brze\'{z}niak, E.~Hausenblas, and P.~A. Razafimandimby.
\newblock Stochastic reaction-diffusion equations driven by jump processes.
\newblock {\em Potential Anal.}, 49(1):131--201, 2018.

\bibitem{DPZ14}
G.~Da~Prato and J.~Zabczyk.
\newblock {\em Stochastic {E}quations in {I}nfinite {D}imensions}, volume 152
  of {\em Encyclopedia of Mathematics and its Applications}.
\newblock Cambridge University Press, Cambridge, second edition, 2014.

\bibitem{DGT11}
A.~Debussche, N.~Glatt-Holtz, and R.~Temam.
\newblock Local martingale and pathwise solutions for an abstract fluids model.
\newblock {\em Phys. D}, 240(14-15):1123--1144, 2011.

\bibitem{DHV16}
A.~Debussche, M.~Hofmanov\'{a}, and J.~Vovelle.
\newblock Degenerate parabolic stochastic partial differential equations:
  quasilinear case.
\newblock {\em Ann. Probab.}, 44(3):1916--1955, 2016.

\bibitem{DM80}
C.~Dellacherie and P.-A. Meyer.
\newblock {\em Probabilit\'{e}s et potentiel. {C}hapitres {V} \`a {VIII}}.
\newblock Actualit\'{e}s Scientifiques et Industrielles [Current Scientific and
  Industrial Topics], No. 1385. Hermann, Paris, revised edition, 1980.

\bibitem{D01}
J.~Droniou.
\newblock {Int{\'e}gration et Espaces de Sobolev {\`a} Valeurs Vectorielles.}
\newblock {\em hal-01382368v2}, 2001.

\bibitem{DG22}
J.~Droniou, B.~Goldys, and K.-N. Le.
\newblock Design and convergence analysis of numerical methods for stochastic
  evolution equations with {L}eray-{L}ions operator.
\newblock {\em IMA J. Numer. Anal.}, 42(2):1143--1179, 2022.

\bibitem{DKL23}
J.~Droniou, M.~A. Khan, and K.~N. Le.
\newblock Numerical analysis of the stochastic {S}tefan problem.
\newblock {\em arXiv:2306.12668}, 2023.

\bibitem{FG95}
F.~Flandoli and D.~G\k{a}tarek.
\newblock Martingale and stationary solutions for stochastic {N}avier-{S}tokes
  equations.
\newblock {\em Probab. Theory Related Fields}, 102(3):367--391, 1995.

\bibitem{GHV22}
M.~V. Gnann, J.~Hoogendijk, and M.~C. Veraar.
\newblock Higher order moments for {SPDE} with monotone nonlinearities.
\newblock {\em arXiv:2203.15307}, 2022.

\bibitem{GRZ09}
B.~Goldys, M.~R\"{o}ckner, and X.~Zhang.
\newblock Martingale solutions and {M}arkov selections for stochastic partial
  differential equations.
\newblock {\em Stochastic Process. Appl.}, 119(5):1725--1764, 2009.

\bibitem{GK96}
I.~Gy\"{o}ngy and N.~Krylov.
\newblock Existence of strong solutions for {I}t\^{o}'s stochastic equations
  via approximations.
\newblock {\em Probab. Theory Related Fields}, 105(2):143--158, 1996.

\bibitem{KR07}
N.~V. Krylov and B.~L. Rozovskii.
\newblock Stochastic evolution equations.
\newblock In {\em Stochastic differential equations: theory and applications},
  volume~2 of {\em Interdiscip. Math. Sci.}, pages 1--69. World Sci. Publ.,
  Hackensack, NJ, 2007.

\bibitem{L69}
J.-L. Lions.
\newblock {\em Quelques m\'{e}thodes de r\'{e}solution des probl\`emes aux
  limites non lin\'{e}aires}.
\newblock Dunod, Paris; Gauthier-Villars, Paris, 1969.

\bibitem{L11}
W.~Liu.
\newblock Existence and uniqueness of solutions to nonlinear evolution
  equations with locally monotone operators.
\newblock {\em Nonlinear Anal.}, 74(18):7543--7561, 2011.

\bibitem{L13}
W.~Liu.
\newblock Well-posedness of stochastic partial differential equations with
  {L}yapunov condition.
\newblock {\em J. Differential Equations}, 255(3):572--592, 2013.

\bibitem{LR10}
W.~Liu and M.~R\"{o}ckner.
\newblock S{PDE} in {H}ilbert space with locally monotone coefficients.
\newblock {\em J. Funct. Anal.}, 259(11):2902--2922, 2010.

\bibitem{LR13}
W.~Liu and M.~R\"{o}ckner.
\newblock Local and global well-posedness of {SPDE} with generalized coercivity
  conditions.
\newblock {\em J. Differential Equations}, 254(2):725--755, 2013.

\bibitem{LR15}
W.~Liu and M.~Röckner.
\newblock {\em Stochastic Partial Differential Equations: An Introduction}.
\newblock Universitext. Springer International Publishing, Cham, 1st ed. 2015
  edition, 2015.

\bibitem{LR17}
S.~V. Lototsky and B.~L. Rozovsky.
\newblock {\em Stochastic Partial Differential Equations}.
\newblock Universitext. Springer International Publishing, Cham, 2017.

\bibitem{MMP14}
C.~Mueller, L.~Mytnik, and E.~Perkins.
\newblock Nonuniqueness for a parabolic {SPDE} with
  {$\frac{3}{4}-\epsilon$}-{H}\"{o}lder diffusion coefficients.
\newblock {\em Ann. Probab.}, 42(5):2032--2112, 2014.

\bibitem{MN15}
L.~Mytnik and E.~Neuman.
\newblock Pathwise uniqueness for the stochastic heat equation with
  {H}\"{o}lder continuous drift and noise coefficients.
\newblock {\em Stochastic Process. Appl.}, 125(9):3355--3372, 2015.

\bibitem{MP11}
L.~Mytnik and E.~Perkins.
\newblock Pathwise uniqueness for stochastic heat equations with {H}\"{o}lder
  continuous coefficients: the white noise case.
\newblock {\em Probab. Theory Related Fields}, 149(1-2):1--96, 2011.

\bibitem{MPS06}
L.~Mytnik, E.~Perkins, and A.~Sturm.
\newblock On pathwise uniqueness for stochastic heat equations with
  non-{L}ipschitz coefficients.
\newblock {\em Ann. Probab.}, 34(5):1910--1959, 2006.

\bibitem{P75}
E.~Pardoux.
\newblock \'{E}quations aux d\'{e}riv\'{e}es partielles stochastiques de type
  monotone.
\newblock In {\em S\'{e}minaire sur les \'{E}quations aux {D}\'{e}riv\'{e}es
  {P}artielles (1974--1975), {III}}, pages Exp. No. 2, 10. Coll\`ege de France,
  Paris, 1975.

\bibitem{P21}
E.~Pardoux.
\newblock {\em Stochastic partial differential equations---an introduction}.
\newblock SpringerBriefs in Mathematics. Springer, Cham, 2021.

\bibitem{RS06}
M.~M. Rao and R.~J. Swift.
\newblock {\em Probability theory with applications}, volume 582 of {\em
  Mathematics and Its Applications (Springer)}.
\newblock Springer, New York, second edition, 2006.

\bibitem{R13}
T.~Roub\'{\i}\v{c}ek.
\newblock {\em Nonlinear partial differential equations with applications},
  volume 153 of {\em International Series of Numerical Mathematics}.
\newblock Birkh\"{a}user/Springer Basel AG, Basel, second edition, 2013.

\bibitem{RSZ22}
M.~Röckner, S.~Shang, and T.~Zhang.
\newblock Well-posedness of stochastic partial differential equations with
  fully local monotone coefficients.
\newblock {\em arXiv:2206.01107}, 2022.

\bibitem{STVZ23}
N.~Sapountzoglou, Y.~Tahraoui, G.~Vallet, and A.~Zimmermann.
\newblock Stochastic pseudomonotone parabolic obstacle problem: well-posedness
  $\&$ {L}ewy-{S}tampacchia's inequalities.
\newblock {\em arXiv:2305.16090}, 2023.

\bibitem{SWZ19}
N.~Sapountzoglou, P.~Wittbold, and A.~Zimmermann.
\newblock On a doubly nonlinear {PDE} with stochastic perturbation.
\newblock {\em Stoch. Partial Differ. Equ. Anal. Comput.}, 7(2):297--330, 2019.

\bibitem{S94}
T.~Shiga.
\newblock Two contrasting properties of solutions for one-dimensional
  stochastic partial differential equations.
\newblock {\em Canad. J. Math.}, 46(2):415--437, 1994.

\bibitem{S87}
J.~Simon.
\newblock Compact sets in the space {$L^p(0,T;B)$}.
\newblock {\em Ann. Mat. Pura Appl. (4)}, 146:65--96, 1987.

\bibitem{S90}
J.~Simon.
\newblock Sobolev, {B}esov and {N}ikol'ski\u{\i} fractional spaces: imbeddings
  and comparisons for vector valued spaces on an interval.
\newblock {\em Ann. Mat. Pura Appl. (4)}, 157:117--148, 1990.

\bibitem{VZ19}
G.~Vallet and A.~Zimmermann.
\newblock Well-posedness for a pseudomonotone evolution problem with
  multiplicative noise.
\newblock {\em J. Evol. Equ.}, 19(1):153--202, 2019.

\bibitem{VZ21}
G.~Vallet and A.~Zimmermann.
\newblock Well-posedness for nonlinear {SPDE}s with strongly continuous
  perturbation.
\newblock {\em Proc. Roy. Soc. Edinburgh Sect. A}, 151(1):265--295, 2021.

\bibitem{W08}
Z.~Wang.
\newblock Existence and uniqueness of solutions to stochastic {V}olterra
  equations with singular kernels and non-{L}ipschitz coefficients.
\newblock {\em Statist. Probab. Lett.}, 78(9):1062--1071, 2008.

\bibitem{W00}
J.~Weidmann.
\newblock {\em Lineare {O}peratoren in {H}ilbertr\"{a}umen. {T}eil 1}.
\newblock Mathematische Leitf\"{a}den. [Mathematical Textbooks]. B. G. Teubner,
  Stuttgart, 2000.
\newblock Grundlagen. [Foundations].

\bibitem{Z09}
X.~Zhang.
\newblock On stochastic evolution equations with non-{L}ipschitz coefficients.
\newblock {\em Stoch. Dyn.}, 9(4):549--595, 2009.

\end{thebibliography}

\end{document}